\documentclass[smallextended,referee,envcountsect,]{svjour3}
\smartqed
\usepackage{graphicx}
\usepackage{subcaption}
\usepackage{amsmath, amssymb}
\usepackage{enumitem}
\usepackage{algorithm}
\usepackage{algpseudocode}
\usepackage{float}
\usepackage{hyperref}
\usepackage{color}
\usepackage{comment}

\begin{document}

\title{On the Worst-Case Analysis of Cyclic Block Coordinate Descent type Algorithms}


\author{Yassine Kamri        \and
        François Glineur \and Julien M. Hendrickx  \and Ion Necoara %etc.
}


\institute{Y. Kamri, F. Glineur, J.M Hendrickx \at INMA, UCLouvain, Louvain-La-Neuve, Belgium.\\ \email{\{yassine.kamri, julien.hendrickx, francois.glineur\}@uclouvain.be}
\and
           I. Necoara \at
              Automatic Control and Systems Engineering Department,  National University of Science and Technology  Politehnica  Bucharest and  Gh. Mihoc-C. Iacob Institute of Mathematical Statistics and Applied Mathematics of the  Romanian Academy,  Bucharest, Romania.\\
              \email{ion.necoara@upb.ro}  
}

\date{Received: date / Accepted: date}

\maketitle

\abstract{We study the worst-case behavior of Block Coordinate Descent (BCD) type algorithms for unconstrained minimization of coordinate-wise smooth convex functions. This behavior is indeed not completely understood, and the practical success of these algorithms is not fully explained by current convergence analyses. We extend the recently proposed Performance Estimation Problem (PEP) approach to convex coordinate-wise smooth functions by proposing necessary interpolation conditions. We then exploit this to obtain improved numerical upper bounds on the worst-case convergence rate of three different BCD algorithms, namely Cyclic Coordinate Descent~\hyperref[alg:CCD]{(CCD)}, Alternating Minimization~\hyperref[alg:AM]{(AM)} and a Cyclic version of the Random Accelerated Coordinate Descent in~\cite{fercoq2015coord}~\hyperref[alg:CACD]{(CACD)}, substantially outperforming the best current bounds in some situations. In addition, we show the convergence of the~\hyperref[alg:CCD]{(CCD)} algorithm with more natural assumptions in the context of convex optimization than those typically made in the literature. Our methodology uncovers a number of phenomena, some of which can be formally established. These include a scale-invariance property of the worst case of~\hyperref[alg:CCD]{(CCD)} with respect to the coordinate-wise smoothness constants and a lower bound on the worst-case performance of~\hyperref[alg:CCD]{(CCD)} which is equal to the number of blocks times the worst-case of full gradient descent over the class of smooth convex functions. We also adapt our framework to the analysis of random BCD algorithms, and
present numerical results showing that the standard acceleration scheme in~\cite{fercoq2015coord} appears to be inefficient for deterministic algorithms.
}

\tableofcontents

\section{Introduction}
Block coordinate-wise descent (BCD) algorithms, which update a single block of coordinates at each update, constitute a widely used class of methods in modern optimization. Indeed, with the widespread availability of data, the scale of modern large-scale optimization problems is constantly increasing, making standard full-gradient optimization methods computationally expensive. Fortunately, many of these problems exhibit a block coordinate structure, making them well-suited for BCD algorithms. These methods can generally be classified into three main categories based on how the blocks of coordinates are selected and updated~\cite{Wright2015}:
\begin{enumerate}
    \item \textbf{Gauss-Southwell methods}, which greedily select the coordinates that lead to the largest improvement.
    \item \textbf{Randomized methods}, which select coordinates according to a probability distribution.
    \item \textbf{Cyclic methods}, which update the coordinates in a predefined cyclic order.
\end{enumerate}
Although greedy methods can perform well, their update rule typically requires access to full gradients to identify which block of coordinates has the largest gradient norm and will lead to the most significant improvement. As a result, randomized and cyclic methods have been more extensively studied and utilized. The worst-case convergence analysis of random coordinate descent methods has proven to be easier than that of their deterministic counterparts. Sampling coordinates with replacement from a suitable probability distribution ensures that the expectation of each coordinate step corresponds to the full gradient, making the analysis largely similar to that of full gradient descent. Consequently, numerous random coordinate descent algorithms with theoretical guarantees have been proposed for convex optimization problems, including accelerated and proximal variants for various probability distributions~\cite{nesterov2012Coords,Lin2015ccd,fercoq2015coord,Diakonikolas2018ccd,Allen2016ccd,Hanzely2019ccd,NesterovStich2017}. However, their performance guarantees hold only in expectation or with high probability, and the sampling technique can be computationally costly specially in high-dimension settings.
 This motivates the study of cyclic block coordinate methods, which are often simpler and more efficient to implement in practice but significantly harder to analyze. The primary challenge lies in establishing a connection between the block of coordinates updated at each step and the full gradient. Some convergence results are already known for cyclic coordinate descent, but they are often obtained under restrictive assumptions, such as the isotonicity of the gradient~\cite{Saha2013cdd}, or with worst-case bounds that are quite conservative and under initial conditions that are not entirely standard for unconstrained convex optimization~\cite{beck2013CCD,Hong2017}. Stronger convergence results exist for some specific cases, such as for quadratic and strongly convex optimization problems~\cite{Sun2015coord,Li2016ccd,Wright2020ccd,Mert2017ccd,Goujaud2022cyclic}.

\medskip

\noindent This paper introduces an extension of the Performance Estimation Problem (PEP) framework, which aims to facilitate the worst-case convergence analysis of block coordinate descent algorithms. The core idea of the PEP framework involves employing semidefinite programming (SDP) to compute convergence guarantees for first-order methods. The original concept of Performance Estimation Problems was first presented in~\cite{drori2014perf}, and subsequent developments were made in~\cite{taylor2017smooth}, where convex interpolation was used to ensure the computation of exact convergence rates. Several extensions of the PEP framework were proposed notably for the analysis of decentralized algorithms~\cite{colla2021}, composite optimization with linear mappings~\cite{Bousselmi2024operators}, adaptive methods~\cite{dasgupta2024nonlinear}, min-max problems~\cite{gorbunov2023minmax}, fixed-points iterations~\cite{park2022fixedpoint}, etc. Another relevant approach to PEP, first proposed in~\cite{lessard2016IQC}, involves conducting a worst-case convergence analysis using principles of control theory and seeking appropriate Lyapunov functions by solving SDPs. In~\cite{taylor2019stochastic}, the authors extend this Lyapunov-based approach to analyze random coordinate descent and provide an optimal convergence bound with respect to the class of Lyapunov functions used.

\medskip

\noindent \textbf{Related work.} As noted previously, cyclic block coordinate descent has received less attention compared to its random counterparts. In the context of unconstrained smooth convex minimization, Beck and Tetruashvili~\cite{beck2013CCD} established a sublinear convergence rate. However, this result relies on the assumption that the set $S = \{ x \in \mathbb{R}^d \mid f(x) \leqslant f(x^{(0)}) \}$ is compact, which implies that the distances between the iterates of a BCD algorithm after each full cycle and the optimal points of the function are bounded. This is a stronger assumption than the one usually made in convex optimization, which only requires that the distance between the first iterate and an optimal point of the function be bounded.
 Additionally, this rate exhibits two behaviors that appear conservative compared to empirical observations: a cubic dependence on the number of blocks $p$ and a dependence on the ratio ${L_{\max}}$/${L_{\min}}$ , where $L_{\max}$ and $L_{\min}$ are the maximal and minimal block coordinate-wise smoothness constants, that would suggests that cyclic coordinate descent performs relatively worse on functions with highly non-uniform coordinate-wise smoothness constants. In~\cite{Wang}, a similar convergence rate is derived and extended to more general settings (proximal methods for composite optimization). In~\cite{Hong2017}, the authors derive a slightly better convergence rate than in~\cite{beck2013CCD} but assume that the objective function admits second-order derivatives.

\medskip

\noindent In~\cite{shi2017coords} the performance estimation problem was used to establish worst-case convergence bounds for cyclic coordinate descent. Although the obtained bound is significantly better than the bound in~\cite{beck2013CCD}, it is established only for block of coordinates of the same dimension and equal block coordinate-wise smoothness constants. Our approach will be valid for arbitrary smoothness constants and is dimension free with respect to each block, i.e. the bounds we obtain here are valid for blocks of arbitrary dimensions. Moreover their bound~\cite[Theorem 3.1]{shi2017coords} contradicts the lower bound on the worst case of cyclic coordinate descent we establish in this paper and thus is false. In~\cite[Appendix I]{taylor2019stochastic}, the authors use an approach similar to Performance Estimation Problems to derive a worst-case convergence bound for random coordinate descent. They achieve this by automatically designing optimal Lyapunov certificates within the considered class of Lyapunov functions by solving a semidefinite program. In a related precedent work presented in~\cite{kamri}, the authors focused on the analysis of block coordinate descent algorithms when applied to (globally) smooth convex functions, and derived upper and lower bounds on the worst-case of these algorithms over the class of convex functions with coordinate-wise Lipschitz gradients. This framework was then slightly improved by~\cite{hadi}, and extended to treat worst-case performances for the specific class of coordinate-wise functions that are also quadratic. In this paper, we present a Performance Estimation Problem framework for analyzing BCD algorithms over the class of coordinate-wise smooth functions. This enables us to derive a tighter worst-case upper bound than the one we obtained in our previous work~\cite{kamri}, where we derived an upper bound on a larger class of globally smooth convex functions that includes the coordinate-wise smooth convex functions generally considered in the literature and in this paper as well. We also improve upon the framework in~\cite{hadi}; indeed, we are able to obtain a tighter worst-case upper bound by using a more suitable characterization of our class of functions in the Performance Estimation Problem.

\medskip

\noindent \textbf{Contributions and Paper Organization.} Our main contributions are as follows:\\
(i) We establish in Section~\ref{sec:intro} that the necessary conditions characterizing the class of block coordinate-wise smooth convex functions in~\cite{shi2017coords} are also sufficient. These conditions are of interest because they simultaneously characterize convexity and block coordinate-wise smoothness, making them particularly suitable for deriving necessary interpolation conditions. We build upon those interpolation conditions to derive a tractable PEP framework for BCD algorithms and show the flexibility of this framework by analyzing in Section~\ref{sec:num_exp} cyclic coordinate descent~\hyperref[alg:CCD]{(CCD)}, alternating minimization~\hyperref[alg:AM]{(AM)} and a cyclic variant of a random accelerated coordinate descent~\hyperref[alg:CACD]{(CACD)} with respect to different initial assumptions.

\noindent (ii) We observe numerically that the worst-case of the cyclic coordinate descent in our setting exhibits some nice properties that we are able to establish formally in Section~\ref{sec:BCD_alg_wc}, namely a scale invariance property with respect to the coordinate-wise smoothness constants and a lower bound on the performance after $K$ cycles of $p$-block cyclic coordinate descent, which equals to $p$ times the exact worst-case of full gradient descent after $pK$ gradient steps. We also derive a new descent lemma for cyclic coordinate descent and a bound on the residual gradient squared norm.

\noindent (iii) We report improved numerical sublinear bounds on the worst-case convergence rate for several types of block coordinate descent (BCD) algorithms. In Section~\ref{sec:num_bound_ccd}, we show that for cyclic block coordinate descent~\hyperref[alg:CCD]{(CCD)}, our bound significantly outperforms the best known analytical result from~\cite{beck2013CCD}. In Section~\ref{sec:AM}, we analyze the alternating minimization algorithm~\hyperref[alg:AM]{(AM)} and observe that the bound from~\cite{beck2013CCD} appears to be asymptotically tight. In Section~\ref{sec:opt_steps}, we provide numerical estimates of the optimal relative step sizes for~\hyperref[alg:CCD]{(CCD)} based on our bounds, which are shorter than those in the full gradient case. We further show in Section~\ref{sec:blocks} that our upper bound on the worst-case convergence rate of CCD grows linearly with the number of blocks. Moreover, in Section~\ref{sec:num_desc}, we adapt our PEP framework to derive an optimization-based lemma that enables the computation of semi-analytical bounds with simple expressions for CCD in a more computationally efficient manner.

\noindent (iv) We introduce in Section~\ref{sec:num_exp} a simplified initial assumption for analyzing cyclic coordinate descent based only on bounding the distance between the initial and optimal points with respect to a scaled norm. Typically, worst-case analysis of BCD algorithms requires stronger initial assumptions that imply bounded distances between all iterates and optimal points. We demonstrate convergence under this simpler assumption and show that the bound we obtain in this setting is only slightly worse than the one we obtain considering the standard initial assumptions used in the literature for BCD algorithms. Moreover under the usual initial assumptions used for the analysis of BCD algorithms, we show that the bound on the worst-case obtained by PEP is significantly better than the best known analytical bounds~\cite{beck2013CCD}.

\noindent (v) We present Section~\ref{sec:CACD} numerical evidence that the acceleration scheme provided in~\cite{fercoq2015coord} for random coordinate descent is less efficient in a deterministic context, i.e., its rate of convergence appears to be slower than $\mathcal{O}(\frac{1}{K^2})$, where $K$ is the number of cycles.

\noindent (vi) In order to show the significance of randomness in the acceleration of BCD algorithms, we extend in Section~\ref{sec:CACD} our PEP framework to handle random algorithms in order to compare their worst-case performance to that of deterministic~ones. We show that in the worst-case the expected performance of random accelerated coordinate descent is better than the performances of all the deterministic variants of the algorithm (i.e all possible orderings of blocks to update).

\section{Introdution and preliminaries}\label{sec:intro}

\subsection{Definition of coordinate-wise smooth functions and BCD algorithms}\label{sec:prob}
We consider the unconstrained convex optimization problem:
\begin{equation*}
    \min_{x \in \mathbb{R}^d} f(x),
\end{equation*}
where \( f: \mathbb{R}^d \to \mathbb{R} \) is a convex differentiable function that satisfies specific block coordinate-wise smoothness properties to be defined below. More specifically, we partition the space \( \mathbb{R}^d \) into \( p \) subspaces: $\mathbb{R}^d = \mathbb{R}^{d_1} \times \dots \times \mathbb{R}^{d_p}$. We introduce the corresponding selection matrices \( U_{\ell} \in \mathbb{R}^{d \times d_{\ell}} \), such that $(U_1, \dots, U_p) = \mathcal{I}_d$. For any \( x \in \mathbb{R}^d \), we can express \( x \) in terms of its block components as:
\begin{equation}\label{eqx_sep_p_blocks}
    x = (x^{(1)}, \dots, x^{(p)})^T, \quad \text{where} \quad x^{(\ell)} = U_{\ell}^{\top}x \in \mathbb{R}^{d_{\ell}}, \quad \forall \ell \in \{1, \dots, p\}.
\end{equation}
Thus, \( x \) can be rewritten as $x = \sum_{\ell=1}^{p} U_{\ell} x^{(\ell)}$.

\begin{definition}
The partial gradients of \( f \) at \( x \) are defined as:
\begin{equation*}
    \nabla^{(\ell)} f(x) \triangleq U_{\ell}^{\top} \nabla f(x) \in \mathbb{R}^{d_{\ell}}, \quad \forall \ell \in \{1, \dots, p\}.
\end{equation*}
\end{definition}

\noindent We now define the class of coordinate-wise smooth convex functions, denoted \( \mathcal{F}^{\text{coord}}_{0,\mathbf{L}}(\mathbb{R}^d) \):

\begin{definition}[Functional class \( \mathcal{F}^{\text{coord}}_{0,\mathbf{L}}(\mathbb{R}^d) \)]\label{def:coord_func}
Given a vector of nonnegative constants \( \mathbf{L} = (L_1, \dots, L_p) \) and a differentiable function \( f:\mathbb{R}^d \mapsto \mathbb{R} \), we say that \( f \) belongs to \( \mathcal{F}^{\text{coord}}_{0,\mathbf{L}}(\mathbb{R}^d) \) if and only if:
\begin{enumerate}
    \item The function $f$ satisfies: \begin{equation*}
        f(x_2) \geq f(x_1) + \langle \nabla f(x_1), x_2 - x_1 \rangle, \quad \forall x_1,x_2 \in \mathbb{R}^d.
    \end{equation*}

    \item For every block of coordinates, \( f \) satisfies the block coordinate-wise smoothness condition:
    \begin{equation*}
        \|\nabla^{(\ell)} f(x + U_{\ell} h^{(\ell)}) - \nabla^{(\ell)} f(x)\| \leq L_{\ell} \|h^{(\ell)}\|, \quad \forall x \in \mathbb{R}^d, \; \forall h^{(\ell)} \in \mathbb{R}^{d_{\ell}}.
    \end{equation*}
\end{enumerate}
\end{definition}
We provide in the next lemma a useful characterization of the functional class $\mathcal{F}^{\text{coord}}_{0,\mathbf{L}}(\mathbb{R}^d)$:
\begin{lemma}\cite[Section 2]{nesterov2012Coords} \label{lm:coord_upper_bound}
Let \( p \) be the number of coordinate blocks, \( \mathbf{L} = (L_1, \dots, L_p) \) a vector of nonnegative constants, and \( f: \mathbb{R}^d \mapsto \mathbb{R} \) a differentiable function. If \( f \in \mathcal{F}^{\text{coord}}_{0,\mathbf{L}}(\mathbb{R}^d) \), then:
\begin{enumerate}
    \item \( f \) satisfies the following inequality:
    \begin{equation*}
    f(x_2) \geq f(x_1) + \langle \nabla f(x_1), x_2 - x_1 \rangle, \quad \forall x_1, x_2 \in \mathbb{R}^d.
    \end{equation*}
    
    \item For all \( \ell \in \{1, \dots, p\} \), \( x \in \mathbb{R}^d \), and \( h^{(\ell)} \in \mathbb{R}^{d_{\ell}} \), \( f \) satisfies the quadratic upper bound:
    \begin{equation}\label{eq:coord_upper_bound}
    f(x + U_{\ell} h^{(\ell)}) \leq f(x) + \langle \nabla^{(\ell)} f(x), h^{(\ell)} \rangle + \frac{L_{\ell}}{2} \|h^{(\ell)}\|^2.
    \end{equation}
\end{enumerate}
\end{lemma}
\noindent For the analysis of first-order algorithms over the functional class \( \mathcal{F}^{\text{coord}}_{0,\mathbf{L}}(\mathbb{R}^d) \), we define the following two norms:

\begin{definition}
Given a vector \( \mathbf{L} = (L_1,\dots,L_p) \) of nonnegative constants, for any \( x,g \in \mathbb{R}^d \), we define the following weighted primal and dual norms:
\begin{equation}\label{eq:coords_norm}
\|x\|^2_L = \sum_{\ell = 1}^p L_\ell \|x^{(\ell)}\|^2 \quad \text{and} \quad 
\|g\|^{*2}_L = \sum_{\ell = 1}^p \frac{1}{L_\ell} \|g^{{(\ell)}}\|^2.
\end{equation}
\end{definition}
Note that the function \( \frac{\|.\|^{*2}_L}{2} \) is the Fenchel conjugate of \( \frac{\|.\|^2_L}{2} \).

\medskip

\noindent We will analyze fixed-step block coordinate-wise algorithms that can be defined as follows
\begin{definition}\label{def:BCD}
Given a number of blocks \( p \), an integer \( N \), and a block of coordinates selection strategy \( t: \{0,\dots,N-1\} \to \{1,\dots,p\} \) that associates each step \( i \) of the algorithm with the index \( t(i) \) of a block of coordinates, we say that \( \mathcal{M}^{\text{coord}} \) is a fixed-step BCD algorithm if, for all \( i \in \{1,\dots,N\} \), the iterate \( x_i \) can be expressed as:
\begin{equation}\label{eq:coord_algo_fixed}
    x_i = x_0 - \sum_{k = 0}^{i-1} \alpha_{i,k} U_{t(k)}\nabla^{(t(k))} f(x_k). 
\end{equation}
We denote by \( \mathcal{M}^{\text{coord}}_{\alpha} \) the fixed-step BCD algorithm parametrized by the step sizes \( \alpha_{i,k} \).
\end{definition}

\begin{remark}
When the block coordinate selection strategy \( t \) is deterministic, for instance for the widely used strategy $t(i) = (i \mod p) + 1$,
which defines a cyclic order selection for the blocks, the BCD algorithm is called \emph{deterministic}. Another possible choice is a random strategy, where the block of coordinates is sampled from a probability distribution. In this case, the BCD algorithm is called \emph{random}.
\end{remark}
One of the most commonly used deterministic block coordinate-wise descent algorithms in practice is the cyclic block coordinate descent algorithm~\hyperref[alg:CCD]{(CCD)}, which performs a partial gradient step at each iteration with respect to a block of coordinates selected in a cyclic order.
\begin{algorithm}[H]
  \caption{Cyclic Coordinate Descent (CCD)}\label{alg:CCD}
\begin{algorithmic}
\State \textbf{Input} function $f$ defined over $\mathbb{R}^d$ with $p$ blocks, starting point $x_0 \in \mathbb{R}^d$, number of cycles $K$ and step-sizes $\{\gamma_\ell\}_{\ell=1}^p$.
\State Define $N = pK$. For $i = 1 \dots N$ ,
\State \hspace{1.5cm} Set $\ell = \text{mod}(i,p) + 1$
\State \hspace{1.5cm} $x_i = x_{i-1} - \gamma_\ell U_\ell\nabla^{(\ell)}f(x_{i-1})$
\end{algorithmic}  
\end{algorithm}

\begin{theorem}\cite[Corollary 3.7]{beck2013CCD}\label{th:coord_beck_ccd}\label{th:ccd_beck} Given a number of blocks of coordinates $p$, a vector of nonnegative constant $\textbf{L} = (L_1, \dots, L_p)$ and a function $f \in \mathcal{F}^{\text{coord}}_{0,\textbf{L}}(\mathbb{R}^d)$, let $\{x_i\}_{i \in \{1,\dots,N\}}$, $N = pK$, be the sequence of iterates generated by $K$ cycles of $p$-block cyclic coordinate descent with step-sizes $\gamma_\ell = \frac{1}{L_\ell}, \; \forall \ell \in \{1,\dots,p\}$, then it holds that:
\begin{equation*}
f(x_{pK})-f_* \leqslant 4L_{max} \left(1+p^3 \frac{L^2_{max}}{L^2_{min}}\right)\frac{1}{K + \frac{8}{p}}R_a^2.
\end{equation*}
where $f_*$ is the minimal value of $f$ and $R_a = \min_{x_* \in X_*} \max_{k \in 1,\dots,K} \|x_{pK}-x_*\|$, $L_{\text{max}} = \max_{\ell \in {1,\dots,p}} L_\ell$ and $L_{\text{min}} = \min_{\ell \in {1,\dots,p}} L_\ell$.
\end{theorem}
\begin{proof}
This is adapted from \cite[Corollary 3.7]{beck2013CCD}. The only difference is that the proof in \cite{beck2013CCD} uses the inequalities:
\begin{equation*}
   \|x_{pk}-x_{*}\| \leqslant \max_{x_* \in X_*} \max_{x \in \mathbb{R}^{d}} \{\|x-x_*\|:\; f(x) \leqslant f(x_0)\}, \quad  \forall k \in 1\dots K.
\end{equation*}
However, the proof remains  valid, considering the optimal point $x_*$ such that
\begin{equation*}
  \|x_{pk}-x_*\| \leqslant \min_{x_* \in X_*} \max_{k \in 1,\dots,K} \|x_{pk}-x_{*}\|,  \quad \forall k \in 1\dots K, \; 
\end{equation*}
and using these inequalities instead. \qed 
\end{proof}

\noindent Here we presented cyclic coordinate descent with a slightly adapted convergence theorem in order to prove that the upper bound remain valid for comparison under the initial assumptions we make in the PEP framework we use to perform worst-case analysis of BCD algorithms (see \texttt{Setting ALL} defined in Section~\ref{sec:num_exp}).

\subsection{Performance Estimation Problems for BCD algorithms}\label{sec:PEP}
The central idea of PEP, first introduced in~\cite{drori2014perf}, is to cast the problem of estimating the worst-case convergence rate of an algorithm as an optimization problem itself. Following this approach, the worst-case performance of the fixed-step first-order algorithm $\mathcal{M}^{\text{coord}}_{\alpha}$ (see Definition~\ref{def:BCD}) over the functional class $\mathcal{F}^{\text{coord}}_{0,\mathbf{L}}(\mathbb{R}^d)$ is given by the following infinite-dimensional problem:

\begin{equation}\tag{i-PEP-coord}\label{i-PEP-coord}
\begin{aligned}
    w^{\text{coord}}(R, \textbf{L} , \mathcal{M}^{\text{coord}}_{\alpha}, N, \mathcal{I}, \mathcal{P}) &= \sup_{f, x_0, \ldots, x_N, x_*} \mathcal{P}(f, x_0, \ldots, x_N, x_*), \\
    \text{such that} \quad &f \in \mathcal{F}^{\text{coord}}_{0,\mathbf{L}}(\mathbb{R}^d), \\
    &x_* \text{ is optimal for } f, \\
    &x_1, \ldots, x_N \text{ generated from } x_0 \text{ by } \mathcal{M}^{\text{coord}}_{\alpha}\\
    &\mathcal{I}(f,x_0,\dots,x_N,x_*) \leq R.
\end{aligned}
\end{equation}
where $\mathcal{P}$ is a performance criterion, for instance $f(x_N)-f(x_*)$ and $\mathcal{I}(f,x_0,\dots,\allowbreak x_N,x_*) \leq R$ is an initial condition that ensures that the worst-case of the algorithm is bounded, for instance $\|x_0 - x_*\| \leq R$. Problem~\eqref{i-PEP-coord} is an infinite-dimensional problem over the functional class $\mathcal{F}^{\text{coord}}_{0,\mathbf{L}}(\mathbb{R}^d)$ which make it hard to solve directly under this form. To alleviate this issue, Taylor et al. introduced in~\cite{taylor2017smooth} the notion of interpolability of a finite set by a class of functions and provided a convex finite-dimensional reformulation of PEP for gradient descent over the class of smooth convex functions. The interpolability of a finite set of triplets by the functional class $\mathcal{F}^{\text{coord}}_{0,\textbf{L}}(\mathbb{R}^d)$ is defined as follows:
\begin{definition}[\( \mathcal{F}^{\text{coord}}_{0,\textbf{L}}(\mathbb{R}^d) \)-interpolability]\label{def:coord_interp}
Let \( I \) be a finite index set, and consider the set of triples \( \mathcal{S} = \{(x_i, g_i, f_i)\}_{i \in I} \), where \( x_i, g_i \in \mathbb{R}^d \) and \( f_i \in \mathbb{R} \) for all \( i \in I \). The set \( \mathcal{S} \) is \( \mathcal{F}^{\text{coord}}_{0,\textbf{L}}(\mathbb{R}^d) \)-interpolable if and only if there exists a function \( f \in \mathcal{F}^{\text{coord}}_{0,\textbf{L}}(\mathbb{R}^d) \) such that $g_i = \nabla f(x_i)$ and $f(x_i) = f_i$, $\forall i \in I$.
\end{definition}
Note that Problem~\ref{i-PEP-coord} only involves first-order information of the functional variable \( f \in \mathcal{F}^{\text{coord}}_{0,\mathbf{L}}(\mathbb{R}^d) \) at the iterates \( x_i, \; i \in \{0,1,\dots,N\} \) of the algorithm \( \mathcal{M}^{\text{coord}}_\alpha \), and at a minimizer \( x_* \) of \( f \). The notion of \( \mathcal{F}^{\text{coord}}_{0,\mathbf{L}}(\mathbb{R}^d) \)-interpolability then allows us to replace the functional variable \( f \) with a finite set of decision variables, $\mathcal{S}_N = \{x_i, g_i, f_i\}_{i \in I} \in (\mathbb{R}^d \times \mathbb{R}^d \times \mathbb{R})^{N+2}$ with $I = \{0,1,\dots,N,*\}$, that is \( \mathcal{F}^{\text{coord}}_{0,\mathbf{L}}(\mathbb{R}^d) \)-interpolable such that Problem~\eqref{i-PEP-coord} takes the finite-dimensional form
\begin{equation}\tag{f-PEP-coord}\label{f-PEP-coord}
\begin{aligned}
    w^{\text{coord}}(R, \textbf{L} , \mathcal{M}^{\text{coord}}_\alpha, N, \mathcal{I}, \mathcal{P})  = &\sup_{\mathcal{S}_N = \{x_i, g_i, f_i\}_{i \in I}} \mathcal{P}(\mathcal{S}_N),\\
     & \text{such that } \mathcal{S}_{N} \text{ is } \mathcal{F}^{\text{coord}}_{0,\mathbf{L}}(\mathbb{R}^d)\text{–interpolable}, \\
     & x_1, \ldots, x_N \text{ generated from } x_0 \text{ by } \mathcal{M}^{\text{coord}}_\alpha, \\
    & g_* = 0, \\
    & \mathcal{I}(\mathcal{S}_N)\leq R.
\end{aligned}
\end{equation}
\subsection{Necessary interpolation conditions}\label{sec:interp}
Now that we have a finite-dimensional formulation for our PEP, we need to characterize in a tractable manner the interpolability constraint on the set $\mathcal{S}_{N}$. To do so, we derive in the next lemma, an alternative characterization of the functional class $\mathcal{F}^{\text{coord}}_{0,\mathbf{L}}(\mathbb{R}^d)$

\begin{lemma}\label{lm:coord_lower_bound}
Let \( p \) be the number of blocks of coordinates, \( \mathbf{L} = (L_1, \dots, L_p) \) a vector of nonnegative constants, and \( f: \mathbb{R}^d \mapsto \mathbb{R} \) a differentiable function. If \( f \in \mathcal{F}^{\text{coord}}_{0,\mathbf{L}}(\mathbb{R}^d) \), then for all \( \ell \in \{1, \dots, p\} \) and for all \( x_1, x_2 \in \mathbb{R}^d \), \( f \) satisfies the quadratic lower bound:
\begin{equation}\label{eq:lower_coord}
f(x_2) \geq f(x_1) + \langle \nabla f(x_1), x_2 - x_1 \rangle + \frac{1}{2L_{\ell}} \|\nabla^{(\ell)} f(x_1) - \nabla^{(\ell)} f(x_2)\|^2.
\end{equation}
\end{lemma}
\noindent The proof for this lemma is delayed in Appendix A.1

\medskip

\noindent We now prove that the characterizations of coordinate-wise smooth convex functions given in Lemmas~\ref{lm:coord_upper_bound} and~\ref{lm:coord_lower_bound} are necessary and sufficient conditions.

\begin{theorem}\label{th:coord_conds}
Let \( p \) be the number of blocks of coordinates, \( \mathbf{L} = (L_1, \dots, L_p) \) a vector of nonnegative constants, and \( f: \mathbb{R}^d \mapsto \mathbb{R} \) a differentiable function. The following statements are equivalent:
\begin{enumerate}
    \item \( f \in \mathcal{F}^{\text{coord}}_{0,\mathbf{L}}(\mathbb{R}^d) \).
    
    \item For all \( x_1, x_2 \in \mathbb{R}^d \),
    \begin{equation*}
        f(x_2) \geq f(x_1) + \langle \nabla f(x_1), x_2 - x_1 \rangle,
    \end{equation*}
    and for all \( \ell \in \{1, \dots, p\} \), \( x \in \mathbb{R}^d \), and \( h^{(\ell)} \in \mathbb{R}^{d_{\ell}} \),
    \begin{equation*}
        f(x + U_{\ell} h^{(\ell)}) \leq f(x) + \langle \nabla^{(\ell)} f(x), h^{(\ell)} \rangle + \frac{L_{\ell}}{2} \|h^{(\ell)}\|^2.
    \end{equation*}
    
    \item For all \( \ell \in \{1, \dots, p\} \) and \( x_1, x_2 \in \mathbb{R}^d \),
    \begin{equation*}
        f(x_2) \geq f(x_1) + \langle \nabla f(x_1), x_2 - x_1 \rangle + \frac{1}{2L_{\ell}} \|\nabla^{(\ell)} f(x_1) - \nabla^{(\ell)} f(x_2)\|^2.
    \end{equation*}
\end{enumerate}
\end{theorem}
\noindent The proof for this Theorem is delayed in Appendix A.2

\medskip

\noindent Compared to the upper-bound characterization in Lemma~\ref{lm:coord_upper_bound}, the lower bound condition~\eqref{eq:lower_coord} provides a more compact characterization of coordinate-wise smooth functions, combining convexity and smoothness for the \( \ell \)-th block of coordinates in a single inequality. Moreover the lower-bound~\eqref{eq:lower_coord} can be written between any pair of points $x_1,x_2 \in \mathbb{R}^d$ whereas the upper bound~\ref{lm:coord_upper_bound} can only be written between two points differing only along the $\ell$-th coordinate. Enforcing the lower bound condition~\eqref{eq:lower_coord} over a set of triplets $\mathcal{S} = \{(x_i, g_i, f_i)\}_{i \in I}$ gives us necessary interpolation conditions for this class $\mathcal{F}^{\text{coord}}_{0,\textbf{L}}(\mathbb{R}^d)$:
\begin{theorem}\label{th:interp_coord}
 Given $I$ a finite index set and a number of blocks of coordinates $p$. If the set $\mathcal{S} = \{(x_i,g_i,f_i\}_{i \in  I} \subset \mathbb{R}^{d} \times \mathbb{R}^{d} \times \mathbb{R}$ is $\mathcal{F}^{\text{coord}}_{0,\textbf{L}}(\mathbb{R}^d)$-interpolable, then
\begin{multline}\label{eq:coord_interp_conds}
    \forall i,j \in I,\; \forall \ell \in \{1,\dots,p\} \\ f_i \geqslant f_j + \sum^p_{\ell = 1} \langle g^{(\ell)}_{j},x^{(\ell)}_{i}-x^{(\ell)}_{j}\rangle + \frac{1}{2L_\ell} \|g^{(\ell)}_{i}-g^{(\ell)}_{j}\|^2.
\end{multline}
\end{theorem}
We discuss in the next theorem the sufficiency of our interpolation conditions
\begin{theorem}\label{th:coord_cond_suff_2}
The interpolation conditions~\eqref{eq:coord_interp_conds} are necessary and sufficient for sets of cardinality $N = 2$, i.e., a set $\mathcal{S} = \{(x_1,g_1,f_1),(x_2,g_2,f_2)\} \subset \mathbb{R}^d \times \mathbb{R}^d \times \mathbb{R}$ is $\mathcal{F}^{\text{coord}}_{0,\textbf{L}}(\mathbb{R}^d)$-interpolable if and only if it satisfies the conditions~\eqref{eq:coord_interp_conds}, but they are not sufficient for sets of triplets of cardinality $N \geq 3$.
\end{theorem}
\noindent The proof for this Theorem is delayed in Appendix A.3

\begin{remark}
Theorems~\ref{th:coord_cond_suff_2} suggests that necessary and sufficient interpolation cannot be enforced by pairwise conditions but instead should result from conditions involving at least $3$ points.
\end{remark}

\noindent By replacing the $\mathcal{F}^{\text{coord}}_{0,\mathbf{L}}(\mathbb{R}^d)$-interpolability constraint on $\mathcal{S}_N$ with the interpolation conditions~\eqref{eq:coord_interp_conds}, we obtain a more tractable relaxation of Problem~\eqref{f-PEP-coord}. Since the class $\mathcal{F}^{\text{coord}}_{0,\mathbf{L}}(\mathbb{R}^d)$ is invariant under additive shifts and translations, we may set $x_* = 0$ and $f^* = 0$ without loss of generality. This leads to:
\begin{equation}\tag{R-PEP-coord}\label{R-PEP-coord}
\begin{aligned}
    w^{\text{coord}}(R, \textbf{L} , \mathcal{M}^{\text{coord}}_\alpha, N, \mathcal{I}, \mathcal{P})  = &\sup_{\mathcal{S}_N = \{x_i, g_i, f_i\}_{i \in I}} \mathcal{P}(\mathcal{S}_N),\\
     & f_i \geqslant f_j + \sum^p_{\ell = 1} \langle g^{(\ell)}_{j},x^{(\ell)}_{i}-x^{(\ell)}_{j}\rangle + \\ &\frac{1}{2L_\ell} \|g^{(\ell)}_{i}-g^{(\ell)}_{j}\|^2,\; \forall i,j \in I,\; \forall \ell \in \{1,\dots,p\} \\
     & x_1, \ldots, x_N \text{ generated from } x_0 \text{ by } \mathcal{M}^{\text{coord}}_\alpha, \\
    & \{x_*,g_*,f_*\} = \{0,0,0\} \\
    & \mathcal{I}(\mathcal{S}_N)\leq R,
\end{aligned}
\end{equation}
where recall that $\mathcal{P}$ is a performance criterion such as $f_N - f_*$ and $\mathcal{I}(\mathcal{S}_N) \leq R$ is an initial condition that ensures that the worst-case of the algorithm is bounded for instance $\mathcal{I}(\mathcal{S}_N) = \|x_0 - x_*\|$. Problem~\eqref{R-PEP-coord} is a relaxation of Problem~\eqref{f-PEP-coord} because the interpolation conditions~\eqref{eq:coord_interp_conds} are proven to be necessary but not sufficient which means that the optimal solution of Problem~\eqref{R-PEP-coord} may not give the exact worst-case convergence rate, but it still provides a valid upper bound on this rate. We show in Section~\ref{sec:cvx_PEP} how Problem~\eqref{R-PEP-coord} can be reformulated as a convex semidefinite program following a similar technique as in~\cite{drori2014perf,taylor2017smooth}.

\section{Worst-case behavior of BCD algorithms}~\label{sec:BCD_alg_wc}

\noindent Before presenting our numerical results obtained using our PEP framework, we establish several general results on the worst-case behavior of cyclic coordinate descent. Specifically, we prove an invariance property of the worst-case performance with respect to the norm $\|.\|_L$ introduced earlier which simplifies the analysis of cyclic coordinate descent using PEP, a lower bound on the worst-case performance of~\hyperref[alg:CCD]{(CCD)} over the class of coordinate-wise smooth convex functions—equal to the number of blocks times the worst-case performance of gradient descent over smooth convex functions—, an improved descent lemma for~\hyperref[alg:CCD]{(CCD)}, and an upper bound on the residual gradient norm for the 2-block~\hyperref[alg:CCD]{(CCD)}.

\medskip

\noindent Usually, the convergence upper bounds derived for first-order algorithms in the literature are independent of the function dimension. Hence, recall that $\mathcal{F}^{\text{coord}}_{0,\mathbf{L}} = \bigcup_{d \in \mathbb{N}} \mathcal{F}^{\text{coord}}_{0,\mathbf{L}}(\mathbb{R}^d)$ is the set of all $\textbf{L}$-block coordinate-wise smooth convex functions, and $\mathcal{F}_{0,L} = \bigcup_{d \in \mathbb{N}} \mathcal{F}_{0,L}(\mathbb{R}^d)$ is the set of all $L$-smooth convex functions. We denote by $\mathcal{W}^{CCD(\gamma)}_{\textbf{L}}(p,K)$ the worst-case performance after $K$ cycles of the $p$-block~\hyperref[alg:CCD]{(CCD)} algorithm with step sizes $\gamma_\ell$ over the set of $\textbf{L}$-block coordinate-wise smooth functions $\mathcal{F}^{\text{coord}}_{0,\textbf{L}}$:
\begin{equation*}
    \mathcal{W}_{\textbf{L}}^{CCD(\gamma)}(p,K) = \max_{f \in \mathcal{F}^{\text{coord}}_{0,\textbf{L}}} \frac{f(x_{pK}) - f_*}{\|x_0 - x_*\|^2_{\textbf{L}}}.
\end{equation*}
We use the scaled norm $\|.\|_\textbf{L}$ in order to obtain and take advantage of an invariance property of the worst-case $\mathcal{W}_{\textbf{L}}^{CCD(\gamma)}$ with respect to vector of smoothness constants $\textbf{L}$. Similarly, we denote by $\mathcal{W}^{GD(\alpha)}_{L}(N)$ the worst-case performance after $N$ steps of gradient descent (GD) with step sizes $\alpha_i$ over the set of $L$-smooth convex functions $\mathcal{F}_{0,L}$:
\begin{equation*}
    \mathcal{W}_{L}^{GD(\alpha)}(N) = \max_{f \in \mathcal{F}_{0,L}} \frac{f(x_N) - f_*}{\|x_0 - x_*\|^2_L},
\end{equation*}
where $\|x_0 - x_*\|^2_L = L \|x_0 - x_*\|^2$.  
\subsection{Scale invariance and Lower Bound on the Worst-case Convergence Rate of (CCD)}
We begin by proving a scale-invariance property for the worst-case behavior of cyclic coordinate descent~\hyperref[alg:CCD]{(CCD)} with respect to the scaled norm $\|.\|_L$. Building on this result, we show that the $p$-block~\hyperref[alg:CCD]{(CCD)} algorithm has a convergence rate on coordinate smooth functions at least $p$ times worse than that of the standard full gradient method on smooth convex functions.

\begin{theorem}\label{th:coord_L_invar}
Given a vector of nonnegative constants $\textbf{L}$, the worst-case performance after $K$ cycles of $p$-block cyclic coordinate descent~\hyperref[alg:CCD]{(CCD)} with step sizes of the form $\{\frac{\gamma_\ell}{L_\ell}\}_{\ell = 1}^p$ over the class $\mathcal{F}^{\text{coord}}_{0,\textbf{L}}$ satisfies the following homogeneity property with respect to the vector of block coordinate-wise smoothness constants $\textbf{L}$:
\begin{equation*}
\mathcal{W}^{CCD(\frac{\gamma_\ell}{L_\ell})}_{\textbf{L}}(p,K) = \mathcal{W}^{CCD(\gamma_\ell)}_{(1,\dots,1)}(p,K)
\end{equation*}  
\end{theorem}
\begin{proof}
Given a function $f$ that belongs to $\mathcal{F}^{\text{coord}}_{0,\textbf{L}}(\mathbb{R}^d)$, define the function $\Tilde{f}: x \mapsto f\left(\frac{1}{\sqrt{L_1}} x^{(1)}, \dots, \frac{1}{\sqrt{L_p}} x^{(p)}\right)$. Then, $\Tilde{f}$ belongs to $\mathcal{F}^{\text{coord}}_{0,(1,\dots,1)}(\mathbb{R}^d)$. Indeed, since $f$ is convex, $\Tilde{f}$ is also convex, and we have:  
\begin{equation*}
    \nabla^{(\ell)} \Tilde{f} (x) = \frac{1}{\sqrt{L_\ell}} \nabla^{(\ell)} f \left(\frac{1}{\sqrt{L_1}} x^{(1)}, \dots, \frac{1}{\sqrt{L_p}} x^{(p)}\right),
\end{equation*}
which, by the $\textbf{L}$-block coordinate-wise smoothness of $f$, implies that for all $x \in \mathbb{R}^d$ and $h^{(\ell)} \in \mathbb{R}^{d_\ell}$,  
\begin{align*}
    \|\nabla^{(\ell)} \Tilde{f} (x + U_{\ell} h^{(\ell)}) - \nabla^{(\ell)} \Tilde{f} (x) \| &\leqslant \frac{1}{\sqrt{L_\ell}} L_\ell \left\|\frac{1}{\sqrt{L_\ell}} h^{(\ell)} \right\| \\
    &= \|h^{(\ell)}\|.
\end{align*}
Thus, $\Tilde{f}$ is $(1,\dots,1)$-block coordinate-wise smooth. Now, denote by $x_i$ the iterates of $CCD(\frac{\gamma_\ell}{L_\ell})$ on $f$ which satisfy $x_i = x_{i-1} - \frac{\gamma_\ell}{L_\ell} U_\ell \nabla^{(\ell)}f(x_{i-1})$, where $\ell$ is chosen cyclically as $\ell = \text{mod}(i,p) + 1$. Let $x_*$ be a minimizer of $f$, and define $\Tilde{x}_{i}$ by setting $\Tilde{x}_i^{(\ell)} = \sqrt{L_\ell} x_i^{(\ell)}$ and $\Tilde{x}_*^{(\ell)} = \sqrt{L_\ell} x_*^{(\ell)}$ for all $\ell \in \{1,\dots,p\}$. Since $x_*$ is a minimizer of $f$, $\Tilde{x}_*$ is a minimizer of $\Tilde{f}$, and it is easy to verify that $\|\Tilde{x}_0 - \Tilde{x}_*\|_{(1,\dots,1)} = \|x_0 - x_*\|_{\textbf{L}}$. Since the iterates $x_i$ satisfy  $x_i = x_{i-1} - \frac{\gamma_\ell}{L_\ell} U_\ell \nabla^{(\ell)} f(x_{i-1})$, it follows that $x_i$ and $x_{i-1}$ only differ in their $\ell^{\text{th}}$ block of coordinates and thus $\Tilde{x}_i$ and $\Tilde{x}_{i-1}$ also only differ in their $\ell^{\text{th}}$ block of coordinates and
\begin{equation*}
\begin{aligned}
 \Tilde{x}_i - \Tilde{x}_{i-1} =   \Tilde{x}^{(\ell)}_i - \Tilde{x}^{(\ell)}_{i-1} = \sqrt{L_{\ell}} (x^{(\ell)}_i - x^{(\ell)}_{i-1}) &= -\sqrt{L_{\ell}} \left(\frac{\gamma_\ell}{L_\ell} U_\ell \nabla^{(\ell)} f(x_{i-1})\right) \\
 &= - \frac{\gamma_\ell}{\sqrt{L_\ell}} U_\ell \nabla^{(\ell)} f(x_{i-1}).
\end{aligned}
\end{equation*}
Using the fact that $\nabla^{(\ell)} \Tilde{f}(\Tilde{x}_{i-1}) = \frac{1}{\sqrt{L_\ell}} \nabla^{(\ell)} f(x_{i-1})$,
we obtain $\Tilde{x}_i = \Tilde{x}_{i-1} - \gamma_\ell U_\ell \nabla^{(\ell)} \Tilde{f}(\Tilde{x}_{i-1})$, which corresponds to the iterates of~\hyperref[alg:CCD]{(CCD)}($\gamma_\ell$) applied to $\Tilde{f}$. Moreover, it is straightforward to verify that $\Tilde{f}(\Tilde{x}_{pK}) - \Tilde{f}(\Tilde{x}_{*}) = f(x_{pK}) - f(x_*)$. Thus, for any function $f \in \mathcal{F}^{\text{coord}}_{0,\textbf{L}}(\mathbb{R}^d)$, we have constructed a function $\Tilde{f} \in \mathcal{F}^{\text{coord}}_{0,(1,\dots,1)}(\mathbb{R}^d)$ such that $\hyperref[alg:CCD]{(CCD)}(\frac{\gamma_\ell}{L_\ell})$ applied to $f$ has the same performance as $\hyperref[alg:CCD]{(CCD)}(\gamma_\ell)$ applied to $\Tilde{f}$. Conversely, for any function $f \in \mathcal{F}^{\text{coord}}_{0,(1,\dots,1)}(\mathbb{R}^d)$, using a similar reasoning, we can show that for $\Tilde{f}: x \mapsto f(\sqrt{L_1} x^{(1)},\dots,\sqrt{L_p} x^{(p)})$, $K$ cycles of $\hyperref[alg:CCD]{(CCD)}(\frac{\gamma_\ell}{L_\ell})$ achieve the same performance on $\Tilde{f}$ as $K$ cycles of $\hyperref[alg:CCD]{(CCD)}(\gamma_\ell)$ on $f$. This finally gives us $\mathcal{W}^{CCD(\frac{\gamma_\ell}{L_\ell})}_{\textbf{L}}(p,K) = \mathcal{W}^{CCD(\gamma_\ell)}_{(1,\dots,1)}(p,K)$, which concludes the proof.
\qed \end{proof}

\begin{remark}
Note that, thanks to Theorem~\ref{th:coord_L_invar}, when using the initial condition $\|x_0 - x_*\|_{\textbf{L}}^2 \leq R_i^2$ in the PEP framework (Setting INIT.), we can, without loss of generality, restrict ourselves to $\textbf{L} = (1,\dots,1)$ and $R_i = 1$.
\end{remark}
\begin{theorem}\label{thm:coord_lower_bound}
Given a number of blocks $p$ and a vector of nonnegative constants $\textbf{L} = (L_1,\dots,L_p)$, the worst-case performance of~\hyperref[alg:CCD]{(CCD)} with step sizes $\{\frac{\gamma_\ell}{L_\ell}\}_{\ell=1}^{p}$ over the class $\mathcal{F}^{\text{coord}}_{0,\textbf{L}}$ is at least $p$ times worse than the worst-case after $pK$ steps of full gradient descent (GD) with step sizes $\alpha_i = \gamma_{i \text{ mod } p + 1}$ over the class $\mathcal{F}_{0,1}$:
\begin{equation*}
    \mathcal{W}^{CCD \left( \frac{\gamma_\ell}{L_\ell} \right)}_{\textbf{L}}(K,p) \geqslant p\mathcal{W}_1^{GD(\gamma_{i \text{ mod }p + 1})}(pK)
\end{equation*}
\end{theorem}
\begin{proof}
The idea of the proof is to construct a function \( \hat{f} \in \mathcal{F}^{\text{coord}}_{0,\mathbf{L}}\) such that the performance of~\hyperref[alg:CCD]{CCD}\(\left(\frac{\gamma_\ell}{L_\ell}\right)\) on \( \hat{f} \) is exactly \( p \) times the worst-case performance of (GD)\((\gamma_{i \bmod p})\) over the class $\mathcal{F}_{0,1}$. Thanks to the scale invariance property in Theorem~\ref{th:coord_L_invar}, we have 
$    \mathcal{W}^{CCD\left(\frac{\gamma_\ell}{L_\ell}\right)}_{\textbf{L}}(p,K) = \mathcal{W}^{CCD(\gamma_\ell)}_{(1,\dots,1)}(p,K)$. Thus, it suffices to prove that $\mathcal{W}^{CCD(\gamma_\ell)}_{(1,\dots,1)}(p,K) \geqslant p\mathcal{W}_1^{GD(\gamma_{i \text{ mod }p+1})}(pK)$. Consider a function $f \in \mathcal{F}_{0,1}$ on which (GD)($\gamma_{i \text{ mod }p+1}$) attains its worst-case performance, and let $d$ be its dimension. Denote by $x_*$ a minimizer of $f$, and choose $x_0 \in \mathbb{R}^d$ such that $\|x_0-x_*\|_1^2 \leqslant R^2$. Now, consider a block-coordinate decomposition of $\mathbb{R}^{pd}$ such that we can write \( x = (x^{(1)}, \dots, x^{(p)})^T \in \mathbb{R}^{pd} \), where each block of coordinates is of size $d$. Define the function \(\hat{f}\) by $\hat{f}(x) = f(x^{(1)} + \dots + x^{(p)})$, for all $x \in \mathbb{R}^{pd}$. Since \( f \) is a $1$-smooth convex function, the function \(\hat{f}\) belongs to \(\mathcal{F}^{\text{coord}}_{0, (1, \dots, 1)}\).  
For all \( x = (x^{(1)}, \dots, x^{(p)})^T \in \mathbb{R}^{pd} \), define \( s(x) = \sum_{\ell=1}^p x^{(\ell)} \), so that for all \( x \in \mathbb{R}^{pd} \), we have $\hat{f}(x) = f(s(x))$. Now, set the initial point for~\hyperref[alg:CCD]{(CCD)}($\gamma_\ell$) on \(\hat{f}\) as \( \hat{x}_0 \in \mathbb{R}^{pd} \) with $\hat{x}^{(\ell)}_0 = \frac{x_0}{p}, \quad \forall \ell \in \{1,\dots,p\}$. Since \( x_* \) is a minimizer of \( f \), the point \(\hat{x}_* \in \mathbb{R}^{pd}\) defined by $\hat{x}^{(\ell)}_* = \frac{x_*}{p}, \quad \forall \ell \in \{1,\dots,p\}$ is a minimizer of \( \hat{f} \). With these definitions, we have \( s(x_0) = x_0 \), \( s(x_*) = x_* \), and  
\begin{equation*}
\begin{aligned}
      \|\hat{x}_0 - \hat{x}_*\|^2_{(1,\dots,1)} &= \sum_{\ell=1}^p \left\|\frac{x_0}{p}-\frac{x_*}{p}\right \|^2 \\
      &= \frac{1}{p^2} \sum_{\ell=1}^p \|x_0-x_*\|^2 = \frac{1}{p} \|x_0-x_*\|_{1}^2 \\
      &= \frac{1}{p} \|s(x_0)-s(x_*)\|_{1}^2.
\end{aligned}
\end{equation*}
At any iteration \( i \) of~\hyperref[alg:CCD]{(CCD)}($\gamma_\ell$), where \( \ell = i \mod p + 1 \), the iterates \( \hat{x}_i \) satisfy the update rule $\hat{x}^{(\ell)}_i = \hat{x}^{(\ell)}_{i-1} - \gamma_\ell \nabla^{(\ell)} \hat{f}(\hat{x}_{i-1})$. This implies that $s(x_i) = s(x_{i-1}) - \gamma_\ell \nabla^{(\ell)} \hat{f}(\hat{x}_{i-1})$. By definition of \( \hat{f} \), we have \( \nabla^{(\ell)} \hat{f}(\hat{x}_{i-1}) = \nabla f(s(x_{i-1})) \), so that $s(x_i) = s(x_{i-1}) - \gamma_\ell \nabla f(s(x_{i-1}))$. This shows that \( \{s(x_i)\} \) corresponds to the iterates of (GD)($\gamma_{i \text{ mod } p + 1}$) on \( f \). Since \( f \) is the worst-case function for (GD)($\gamma_{i \text{ mod } p + 1}$), we have that  
\begin{equation*}
    \mathcal{W}_1^{GD(\gamma_{i \text{ mod }p})}(pK) = \frac{f(s(x_{pK})) - f(s(x_*))}{\|s(x_0)-s(x_*)\|_1^2}.
\end{equation*}
Since \( \hat{f}(\hat{x}_i) = f(s(x_i)) \) and \( \|\hat{x}_i - \hat{x}_*\|^2_{(1,\dots,1)} = \frac{1}{p} \|s(x_0)-s(x_*)\|_1^2 \), we deduce that the performance of~\hyperref[alg:CCD]{(CCD)}($\gamma_\ell$) after \( K \) cycles on \( \hat{f} \) is given by  
\begin{equation*}
    \frac{\hat{f}(\hat{x}_{pK}) - \hat{f}(\hat{x}_{*})}{\|\hat{x}_{0} - \hat{x}_{*}\|^2_{(1, \dots, 1)}} = p \frac{f(s_{pK}) - f(s(x_*))}{\|s(x_0) - s(x_*)\|_1^2} = p \mathcal{W}_1^{\text{GD}(\gamma_{i \bmod p})}(pK).
\end{equation*}
By the definition of the worst-case performance of~\hyperref[alg:CCD]{(CCD)}($\gamma_\ell$) over \( \mathcal{F}^{\text{coord}}_{0,(1,\dots,1)} \), we conclude that  
$\mathcal{W}^{CCD(\gamma_\ell)}_{(1,\dots,1)}(p,K) \geqslant p\mathcal{W}_1^{GD(\gamma_{i \text{ mod }p+1})}(pK)$. This completes the proof.
\qed \end{proof}

\begin{remark}
The previous theorem can be generalized to any deterministic strategy for selecting the block of coordinates to update, other than the cyclic order.  
\end{remark}

\begin{remark}
The upper bound on the convergence of~\hyperref[alg:CCD]{(CCD)} derived in \cite[Theorem 1]{shi2017coords} (also using PEP) violates the lower bound established in Theorem~\ref{thm:coord_lower_bound} and thus appears to be incorrect. 
\end{remark}
\subsection{Descent Lemmas for BCD Algorithms}\label{sec:desc_lemma_2}
\noindent  We now provide an improved descent lemma for the $2$-block~\hyperref[alg:CCD]{(CCD)} and derive from this lemma a bound on the residual gradient norm.  
\begin{lemma}[Descent lemma for 2-block (CCD)]\label{lm:coord_lemma_descent}
Given two successive iterates \( x_0 = (x^{(1)}_0, x^{(2)}_0)^T \) and \( x_1 = (x^{(1)}_{1}, x^{(2)}_{1})^T \) of~\hyperref[alg:CCD]{(CCD)}\(\left(\frac{1}{L_\ell}\right)\) on a function \( f \) that belongs to \( \mathcal{F}^{\text{coord}}_{0,\textbf{L}}(\mathbb{R}^d) \), where the update is performed on the block \( \ell \in \{1,2\} \):  
\begin{equation}
\label{eq: step_i}
    x_{1} = x_{0} - \frac{1}{L_\ell} U_{\ell} \nabla^{(\ell)} f(x_0),
\end{equation}
we have:
\begin{equation}\label{desc_lemma}
    f(x_0)-f(x_1) \geqslant \frac{1}{2L_\ell} \left (\|\nabla^{(\ell)} f(x_0)\|^2 + \| \nabla^{(\ell)} f(x_{1})\|^2 \right).
\end{equation}
\end{lemma}

\begin{proof}
Since \( f \) belongs to \( \mathcal{F}^{\text{coord}}_{0,\textbf{L}}(\mathbb{R}^d) \), Theorem~\ref{th:coord_conds} gives:
\begin{equation*}
    f(x_0) - f(x_1) \geqslant \langle \nabla f(x_1), x_0 - x_1\rangle + \frac{1}{2L_\ell} \|\nabla^{(\ell)} f(x_1) - \nabla^{(\ell)} f(x_0)\|^2.
\end{equation*}
Since \( x_0 \) and \( x_1 \) differ only in the \( \ell \)-th block of coordinates, we have $\langle \nabla f(x_1), x_0 - x_1\rangle = \langle \nabla^{(\ell)} f(x_1), x^{(\ell)}_0 - x^{(\ell)}_1\rangle$, leading to:
\begin{equation}
\label{eq: th1}
    f(x_0) - f(x_1) \geqslant \langle \nabla^{(\ell)} f(x_1), x^{(\ell)}_0 - x^{(\ell)}_1\rangle + \frac{1}{2L_\ell} \|\nabla^{(\ell)} f(x_1) - \nabla^{(\ell)} f(x_0)\|^2.
\end{equation}  
Substituting~\eqref{eq: step_i} into~\eqref{eq: th1} yields
\begin{equation*}
f(x_0) - f(x_1) \geqslant \frac{1}{L_\ell}\langle \nabla^{(\ell)} f(x_1), \nabla^{(\ell)} f(x_0)\rangle + \frac{1}{2L_\ell} \|\nabla^{(\ell)} f(x_1) - \nabla^{(\ell)} f(x_0)\|^2.
\end{equation*}
Since
\begin{equation*}
\begin{aligned}
\frac{1}{2L_\ell} \|\nabla^{(\ell)} f(x_1) - \nabla^{(\ell)} f(x_0)\|^2 = \frac{1}{2L_\ell} \|\nabla^{(\ell)} f(x_1)\|^2 &\\- \frac{1}{L_\ell} \langle \nabla^{(\ell)} f(x_1), \nabla^{(\ell)} f(x_0)\rangle + \frac{1}{2L_\ell} \|\nabla^{(\ell)} f(x_0)\|^2 
\end{aligned}
\end{equation*}
we have
\begin{equation*}
\begin{aligned}
     \langle \nabla^{(\ell)} f(x_1), x^{(\ell)}_0 - x^{(\ell)}_1\rangle + \frac{1}{2L_\ell} \|\nabla^{(\ell)} f(x_1) - \nabla^{(\ell)} f(x_0)\|^2 = & \\ \frac{1}{2L_\ell}\left (\|\nabla^{(\ell)} f(x_0)\|^2 + \| \nabla^{(\ell)} f(x_{1})\|^2 \right).
\end{aligned}
\end{equation*}
Plunging this into~\eqref{eq: th1} gives us the desired result and concludes the proof.
\qed \end{proof}

\begin{remark}
Lemma~\ref{lm:coord_lemma_descent} is analogous to the double sufficient decrease lemma derived in \cite{teboulle2023elementary} for full gradient descent.
\end{remark}

\begin{theorem}
Given a vector of nonnegative constants $\textbf{L}$ and a function $f$ that belongs to \(\mathcal{F}^{\text{coord}}_{0,\textbf{L}}(\mathbb{R}^d) \), after \( K \) cycles of the $2$-block~\hyperref[alg:CCD]{(CCD)}\(\left(\frac{1}{L_\ell}\right)\) applied on $f$, we have:
\begin{equation}\label{eq:coord_bound_desc}
    \min_{i \in [1,2K-1]} \|\nabla f(x_i)\|_\textbf{L}^{*2} \leqslant \frac{2}{2K-1} \left(f(x_0) - f(x_{2K})\right).
\end{equation}
\end{theorem} 

\begin{proof}
Consider one cycle of $2$-block coordinate descent. Without loss of generality, we can order the partial gradient updates as $ x_1 = x_0 - \frac{1}{L_1} U_1\nabla^{(1)} f(x_0)$ and $ x_2 = x_1 - \frac{1}{L_2} U_2\nabla^{(2)} f(x_{1})$. By Lemma~\eqref{lm:coord_lemma_descent}, we have $f(x_0) - f(x_1) \geqslant \frac{1}{2L_1} (\|\nabla^{(1)} f(x_1)\|^2 + \|\nabla^{(1)} f(x_0)\|^2)$
and \( f(x_1) - f(x_2) \geqslant \frac{1}{2L_2} ( \allowbreak \|\nabla^{(2)} f(x_1)\|^2 + \allowbreak \|\nabla^{(2)} f(x_2)\|^2 )\). Summing the two inequalities gives:
\begin{align*}
       f(x_0) - f(x_2) & \geqslant \frac{1}{2L_1}\|\nabla^{(1)} f(x_0)\|^2  + \frac{1}{2} \|\nabla f(x_1)\|_{\textbf{L}}^{*2} +  \frac{1}{2L_2}\|\nabla^{(2)} f(x_2)\|^2\\
       &\geqslant  \frac{1}{2} \|\nabla f(x_1)\|_\textbf{L}^{*2}, 
\end{align*}
which corresponds exactly to \eqref{eq:coord_bound_desc} for one cycle (\( K = 1 \)). The proof easily generalizes to \( K \) cycles. By writing all the valid inequalities~\eqref{desc_lemma} between consecutive iterates and summing them, we obtain:
\begin{equation*}
\begin{aligned}
f(x_0) - f(x_{2K}) &\geqslant \frac{1}{2L_1} \|\nabla^{(1)} f(x_0)\|^2 + \frac{1}{2} \sum_{i = 1}^{2K - 1} \|\nabla f(x_i)\|_{\textbf{L}}^{*2} \\
&+  \frac{1}{2L_2} \|\nabla^{(2)} f(x_{2K})\|^2.
\end{aligned}
\end{equation*}
Since:
\begin{align*}
\frac{1}{2L_1} \|\nabla^{(1)} f(x_0)\|^2 + \frac{1}{2} \sum_{i = 1}^{2K - 1} \|\nabla f(x_i)\|_{\textbf{L}}^{*2} +  \frac{1}{2L_2} \|\nabla^{(2)} f(x_{2K})\|^2  & \geqslant \\ \frac{2K-1}{2} \min_{i \in [1,2K-1]} \|\nabla f(x_i)\|_\textbf{L}^{*2},
\end{align*}
we obtain the desired result.
\qed \end{proof}
\noindent Note that the previous theorem is not trivial to generalize in an analytic form for a number of blocks \( p > 2 \); however, in Section~\ref{sec:num_desc}, we show how Performance Estimation Problems can be used to numerically generate optimized descent lemmas and then derive from these semi-analytical upper bounds on the convergence of~\hyperref[alg:CCD]{(CCD)}.

\section{Numerical Results on the Worst-case Convergence of BCD Algorithms}\label{sec:num_exp}
\subsection{Convex reformulation for PEPs and Initial Assumptions}\label{sec:cvx_PEP}
We now exploit the PEP framework developed previously to provide new improved numerical bounds on the worst-case convergence rate of deterministic block coordinate descent algorithms and give new insights on the worst-case behaviour of these algorithms. \noindent Under its current form, Problem~\eqref{R-PEP-coord} remains nonconvex (interpolation conditions~\eqref{eq:coord_interp_conds} involve nonconvex quadratic terms). However, for any fixed-step first-order BCD algorithm $\mathcal{M}^{\text{coord}}_\alpha$ such as for instance cyclic coordinate descent, Problem~\eqref{R-PEP-coord} can be reformulated as a convex SDP. Indeed, for any deterministic fixed-step first-order BCD algorithm, as defined by equation~\eqref{eq:coord_algo_fixed}, the iterates $\{x_i\}$ can be expressed as linear functions of the initial point $x_0$ and past partial gradients $g^{(\ell)}_i$. By defining the matrices $P^{(\ell)} = [g^{(\ell)}_0,\dots,g^{(\ell)}_N,x_0^{(\ell)}],\; \forall \ell \in \{1,\dots,p\}$, any constraint involving the quadratic expressions $\langle x^{(\ell)}_i, x^{(\ell)}_j\rangle, \; \langle g^{(\ell)}_i, x^{(\ell)}_j\rangle, \; \langle g^{(\ell)}_i, g^{(\ell)}_j\rangle, \allowbreak \; \forall i,j \in I, \; \forall \ell \in \{1,\dots,p\}$ can be rewritten as a linear constraint on the Gram matrices $G^{(\ell)} = P^{(\ell)\top} P^{(\ell)}$. Standard performance criteria such as $f_N -f_*$ and initial conditions such as $\|x_0-x_*\| \leq R$ can also be expressed as linear expression of the vector of function values $F = [f_1,\dots,f_N]$ and the Gram matrices $G^{(\ell)}$. For a detailed explanation of the Gram matrix lifting technique used to reformulate a PEP as an SDP, we refer the reader to~\cite{taylor2017smooth}, where it is presented in the context of full gradient descent over the class of smooth (strongly) convex functions. The main difference with the case of  full gradient descent over the class of smooth (strongly) convex functions is that the analysis of $K$ cycles of a $p$-block coordinate wise algorithm requires solving an SDP involving $p$ matrices instead of only one of size $(N+2) \times (N+2)$ with $N = pK$, the number of partial gradient steps which can prove to be computationally expensive specially for large numbers of cycles $K$ or blocks of coordinates $p$.

\noindent In the remainder of this paper, we choose for the analysis of our BCD algorithms the functional accuracy \( f(x_N) - f(x_*) \) as the performance criterion \( \mathcal{P} \). Regarding the initial condition, a common assumption in the analysis of cyclic block coordinate-wise descent algorithms, as adopted in~\cite{beck2013CCD,nesterov2012Coords}, is that the level set $S = \{x \in \mathbb{R}^d : f(x) \leq f(x_0)\}$ is compact, which implies the following bounds (assuming that that the algorithm decreases the objective monotonically which is the case for instance for~\hyperref[alg:CCD]{(CCD)} ):

\begin{equation}\label{eq:coord_init_beck}
    \|x_{pk} - x_*\| \leq R(x_0), \quad \forall k \in \{1,\dots,K\},
\end{equation}
where \( R(x_0) \) is defined as:
\begin{equation}
    R(x_0) = \max_{x_* \in X_*} \max_{x \in \mathbb{R}^d} \{\|x-x_*\| : f(x) \leq f(x_0)\},
\end{equation}
\noindent with $X_*$ is the set of optimal points of the function. As explained in Theorems~\ref{th:coord_beck_ccd} and~\ref{th:coord_am_bound}, the upper bounds in~\cite{beck2013CCD} remain valid considering the optimal point $x_*$ such that $\forall k \in 1\dots K, \;  \|x_{pk}-x_*\| \leqslant \min_{x_* \in X_*} \max_{k \in 1,\dots,K} \|x_{pk}-x_{*}\|$, This allows us to define a Setting ALL such that Theorems~\ref{th:coord_beck_ccd} and~\ref{th:coord_am_bound} remain applicable and their bounds remain valid for comparison.

\medskip

\noindent \textbf{Setting ALL.} Given the parameter $R_a>0$, we assume that the iterates satisfy:
\begin{equation}\label{eq:coord_setting_all}
    \min_{x_* \in X_*} \max_{k \in \{1,\dots,K\}} \|x_{pk} - x_*\| \leq R_a.
\end{equation}
Note that the condition~\eqref{eq:coord_setting_all} in Setting ALL can be used as an initial condition in our PEP framework, as it is equivalent to imposing $\|x_{pk} - x_*\| \leq R_a, \quad \forall k \in \{1,\dots,K\}$, which can be expressed as linear constraints on the Gram matrices $G^{(\ell)}$.

\medskip

\noindent Though theoretically convenient, Setting ALL is not very natural and may be difficult to verify in practice. In addition, Setting ALL can prove to be unusable for certain class of functions. Indeed, consider the family of smooth functions $f_{\epsilon}(x,y) = (x - y)^2 + \epsilon (x^2 + y^2)$ and the initial point $(x_0 = 1, y_0 = -1)$ for cyclic coordinate descent. We will show for this class of functions that $ R_a = \min_{x_* \in X_*} \max_{k \in \{1,\dots,K\}} \|x_{pk} - x_*\|$ can become very large and therefore any bound  proportional to $R_a^2$ (which is the case for the bounds derived under the assumption of Setting ALL) also become very large. Computing the gradients along each of the two blocks of coordinates gives
\begin{equation*}
\begin{aligned}
    &\forall x, y \in \mathbb{R}^d, \; \nabla^{(x)} f(x,y) = 2(1+\epsilon) x - 2y, \\
    &\forall x, y \in \mathbb{R}^d, \; \nabla^{(y)} f(x,y) = 2(1+\epsilon) y - 2x. \\
\end{aligned}
\end{equation*}

\noindent This implies that $f_\epsilon$ is $2(1+\epsilon)$-smooth along each coordinate. Applying the 2-block cyclic coordinate descent algorithm with the step size $\frac{1}{2(1+\epsilon)}$, starting along $x$, we obtain that after one cycle of the algorithm $x_{p(k+1)} = \frac{y_{pk}}{(1+\epsilon)}$ and $y_{p(k+1)} = \frac{y_{pk}}{(1+\epsilon)^2}$. We have that $R_a = \min_{x_* \in X_*} \max_{k \in \{1,\dots,K\}} \|x_{pk} - x_*\| = \|(x_2, y_2)\| = \frac{1}{(1+\epsilon)^2} \sqrt{1 + \frac{1}{(1+\epsilon)^2}}$, which tends to infinity as $\epsilon$ tends to zero. For $\epsilon$ small enough the bounds obtained in this setting are thus very conservative and do not give useful information about the performance of the algorithm. Therefore, we will also consider a more classical setting, albeit less frequently used in the context of cyclic block coordinate-wise algorithms:

\medskip

\noindent \textbf{Setting INIT.} Given $R_i>0$, the starting point of the block coordinate algorithm $x_{0}$ and an optimal point of the function $x_*$, we have that
\begin{equation}\label{eq:setting_init}
     \|x_{0}-x_*\|_\textbf{L} \leqslant R_i.
\end{equation}
The condition~\eqref{eq:setting_init} in Setting INIT is also compatible with our PEP framework for BCD algorithms as it can be expressed as a linear constraint on the Gram matrices $G^{(\ell)}$.

\subsection{New worst-case convergence rate upper bounds for~\hyperref[alg:CCD]{(CCD)} in Setting ALL and INIT }\label{sec:num_bound_ccd}

First, we consider the algorithm~\hyperref[alg:CCD]{(CCD)}$\left(\gamma_\ell\right)$ with step sizes $\gamma_\ell = \frac{1}{L_\ell}$ for all $\ell \in \{1,\dots,p\}$. We provide bounds for cases with $2$ and $3$ blocks and a range of Lipschitz constants. Although we present results only for $2$- and $3$-block algorithms, our PEP framework can accommodate an arbitrary number of blocks of coordinates. As a performance criterion, we measure the difference between the objective value at the final iterate and the optimal value, $f(x_{pK}) - f(x_*)$. Since a sublinear convergence rate of $\mathcal{O}(1/K)$ is expected, we present our results in the form of the bound multiplied by the number of cycles, to facilitate analysis. Figures~\ref{fig:CCD_all_blocks} show that the obtained upper bounds outperform by one order of magnitude the best-known analytical bound given in Theorem~\ref{th:coord_beck_ccd}. Additionally, these figures illustrate that our upper bounds improve upon those obtained using the PEP framework from~\cite{hadi}, which uses the inequalities of $2$ in Theorem~\ref{th:coord_conds}---namely, the classical subgradient inequality to characterize convexity and the upper bound characterization of coordinate-wise smoothness---as interpolation conditions, demonstrating that our framework provides a tighter analysis due to a better characterization (interpolation conditions) of the functional class $\mathcal{F}^{\text{coord}}_{0,\textbf{L}}(\mathbb{R}^d)$.
 We also compare our upper bounds to the lower bounds obtained using the PEP framework developed in~\cite{kamri} for BCD algorithms over the class of smooth convex functions which tends to indicate that our bounds are almost tight.

 \medskip

 \noindent In Figure~\ref{fig:CCD_init}, we provide numerical upper bounds for the $2$-block~\hyperref[alg:CCD]{(CCD)}$\left( \frac{1}{L_\ell}\right)$ and $3$-block~\hyperref[alg:CCD]{(CCD)}$\left( \frac{1}{L_\ell}\right)$ under the milder assumptions of Setting INIT. By selecting step-sizes $\gamma_{\ell} = \frac{1}{L_{\ell}}$ for $\ell \in \{1,\dots,p\}$, we can assume without loss of generality that $\textbf{L} = (1,\dots,1)$ and $R_{i} = 1$ due to the scaling law stated in Theorem~\ref{th:coord_L_invar}. We plot lower and upper bounds on the function accuracy after $K$ cycles, $f(x_{pK}) - f(x_*)$, times the number of cycles $K$. We compare our upper bound to the one obtained using the PEP framework developed in~\cite{hadi}, as well as to the lower bound derived in Theorem~\ref{thm:coord_lower_bound}. As expected, the upper bound on the worst case obtained using our PEP framework outperforms the upper bound obtained using the PEP framework from~\cite{hadi}, as our interpolation conditions are tighter. We also provide a linear fit of the reciprocal of our bounds, suggesting that the convergence rate of~\hyperref[alg:CCD]{(CCD)} under the assumptions of Setting INIT is $\mathcal{O}(1/K)$. We conclude that the stronger assumptions of Setting ALL do not appear to yield faster convergence compared to the milder assumptions of Setting INIT. To the best of our knowledge, no analytical convergence bound exists for deterministic BCD algorithms under the milder assumptions of Setting INIT. We thus provide the first numerical evidence suggesting convergence of~\hyperref[alg:CCD]{(CCD)} in this setting, and with order $\mathcal{O}(1/K)$.

\begin{figure}[H]
    \centering
    % --- Row 1: 3-block CCDs ---
    \begin{subfigure}[b]{0.4\textwidth}
        \centering
        \includegraphics[width=\textwidth]{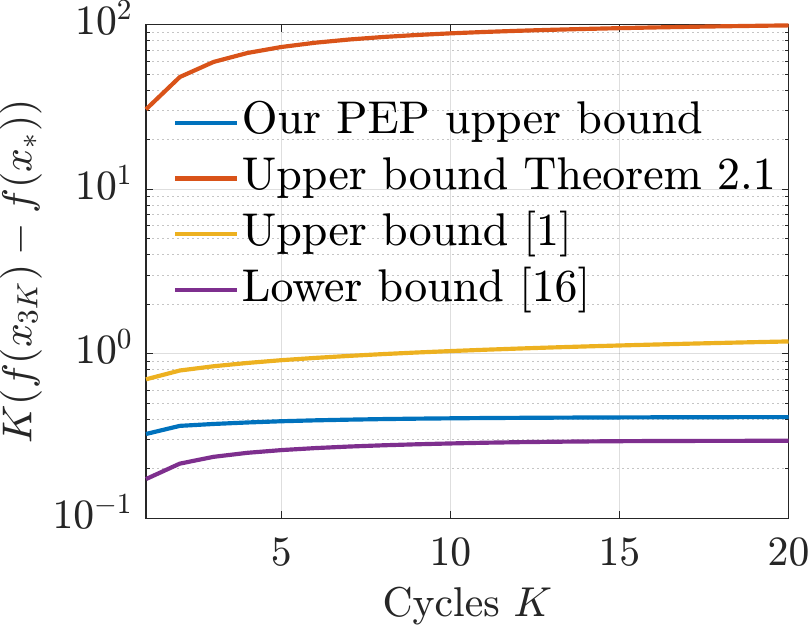}
        \caption{$\mathbf{L} = (1,1,1)$}
        \label{fig:CCD311}
    \end{subfigure}
    \hspace{0.001\textwidth}
    \begin{subfigure}[b]{0.4\textwidth}
        \centering
        \includegraphics[width=\textwidth]{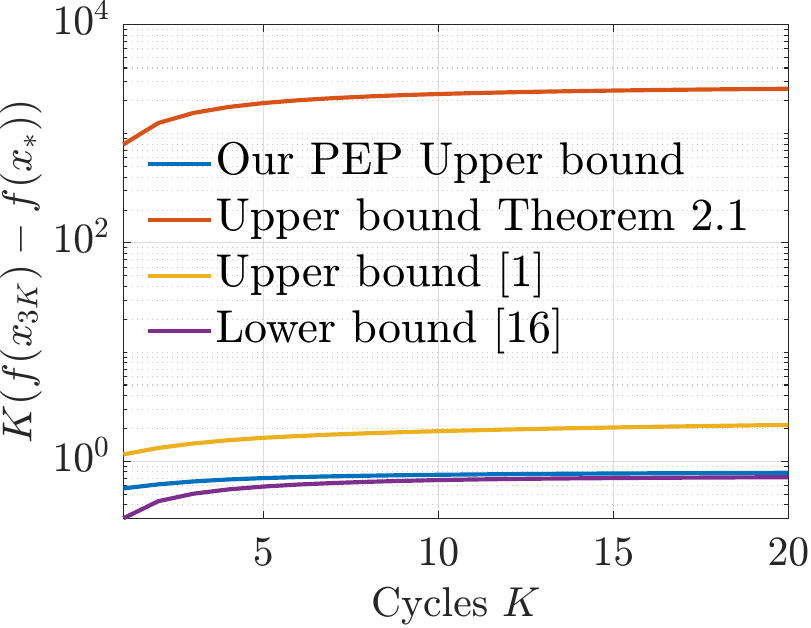}
        \caption{$\mathbf{L} = (1,1,3)$}
        \label{fig:CCD3113}
    \end{subfigure}
    \hspace{0.001\textwidth}
    \begin{subfigure}[b]{0.4\textwidth}
        \centering
        \includegraphics[width=\textwidth]{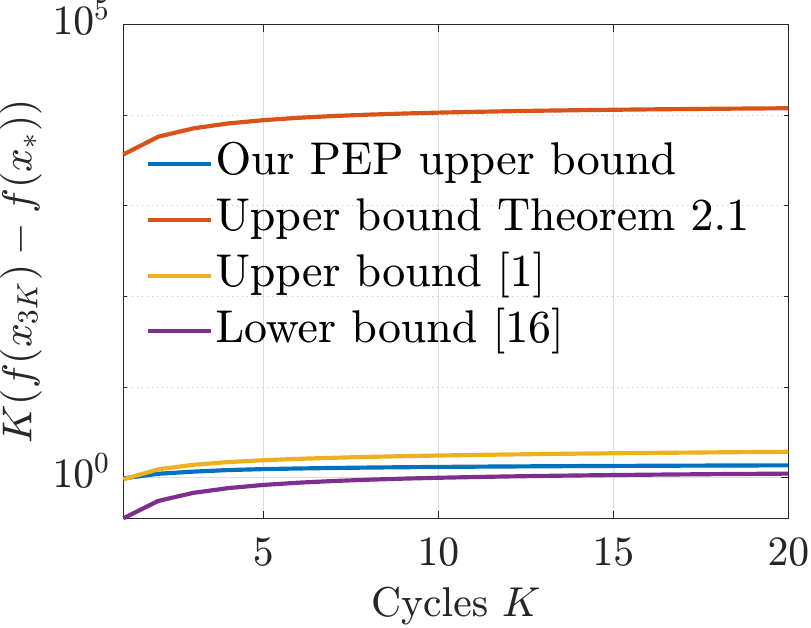}
        \caption{$\mathbf{L} = (1,3,5)$}
        \label{fig:CCD3135}
    \end{subfigure}

    %\vspace{-0.6em}

    % --- Row 2: 2-block CCDs ---
    \begin{subfigure}[b]{0.4\textwidth}
        \centering
        \includegraphics[width=\textwidth]{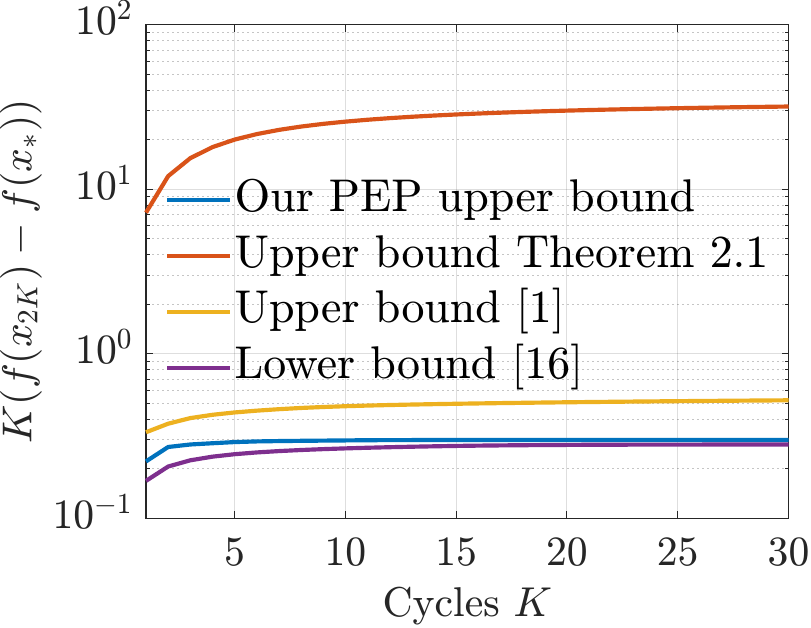}
        \caption{$\mathbf{L} = (1,1)$}
        \label{fig:CCD211}
    \end{subfigure}
    \hspace{0.001\textwidth}
    \begin{subfigure}[b]{0.4\textwidth}
        \centering
        \includegraphics[width=\textwidth]{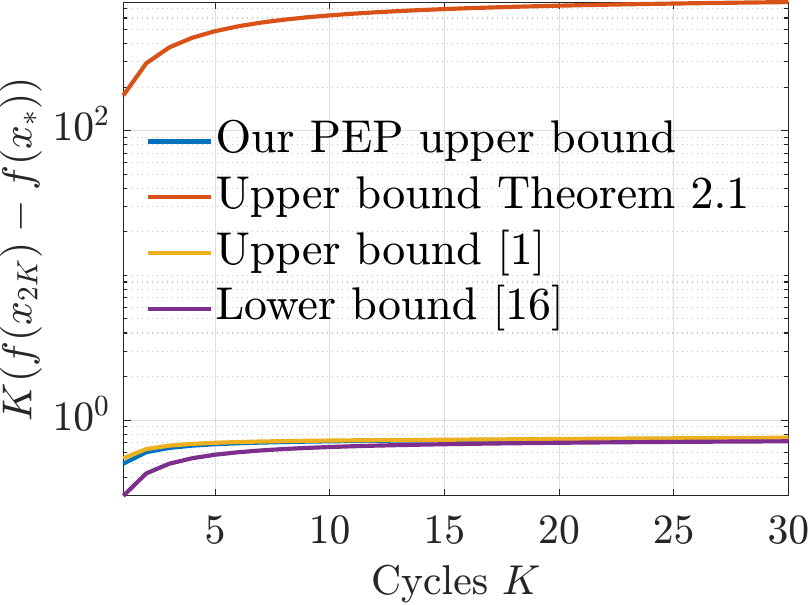}
        \caption{$\mathbf{L} = (1,3)$}
        \label{fig:CCD213}
    \end{subfigure}
    \hspace{0.001\textwidth}
    \begin{subfigure}[b]{0.4\textwidth}
        \centering
        \includegraphics[width=\textwidth]{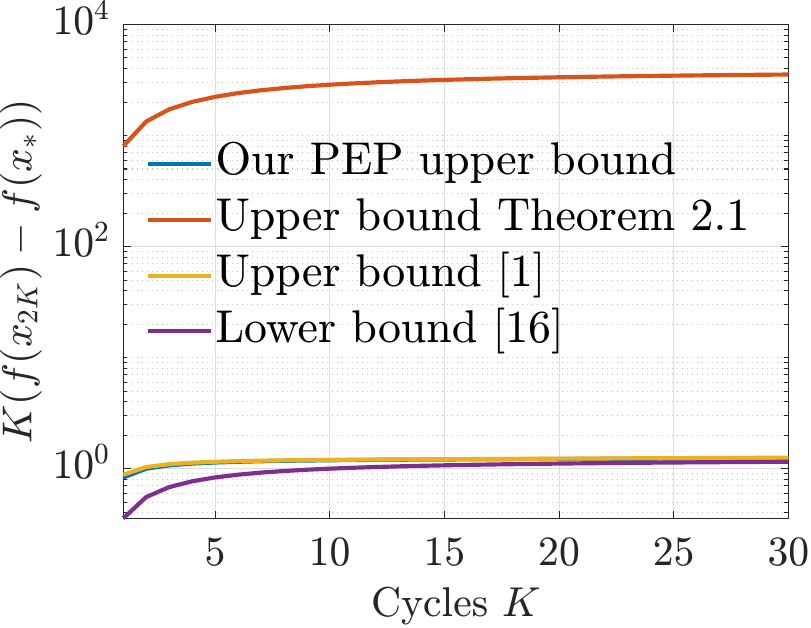}
        \caption{$\mathbf{L} = (1,5)$}
        \label{fig:CCD215}
    \end{subfigure}

    \caption{\small Comparison of upper bounds on the worst-case performance of 2-block and 3-block~\hyperref[alg:CCD]{(CCD)} methods in the Setting ALL.}
    \label{fig:CCD_all_blocks}
\end{figure}

\begin{figure}[H]
    \centering

    % --- Row 1: INIT bounds ---
    \begin{subfigure}[b]{0.45\textwidth}
        \centering
        \includegraphics[width=\textwidth]{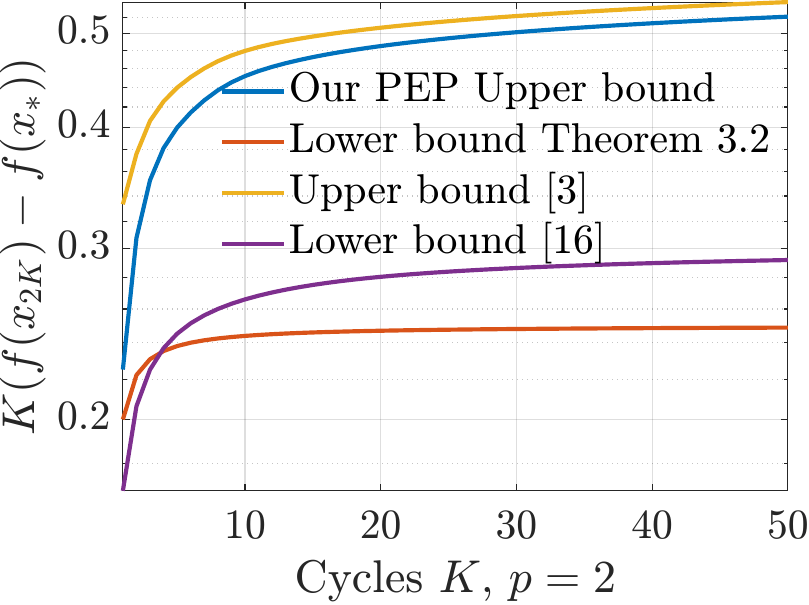}
        \caption{$2$-block~\hyperref[alg:CCD]{(CCD)} bounds}
        \label{fig:CCD2_init}
    \end{subfigure}
    \hspace{0.001\textwidth}
    \begin{subfigure}[b]{0.45\textwidth}
        \centering
        \includegraphics[width=\textwidth]{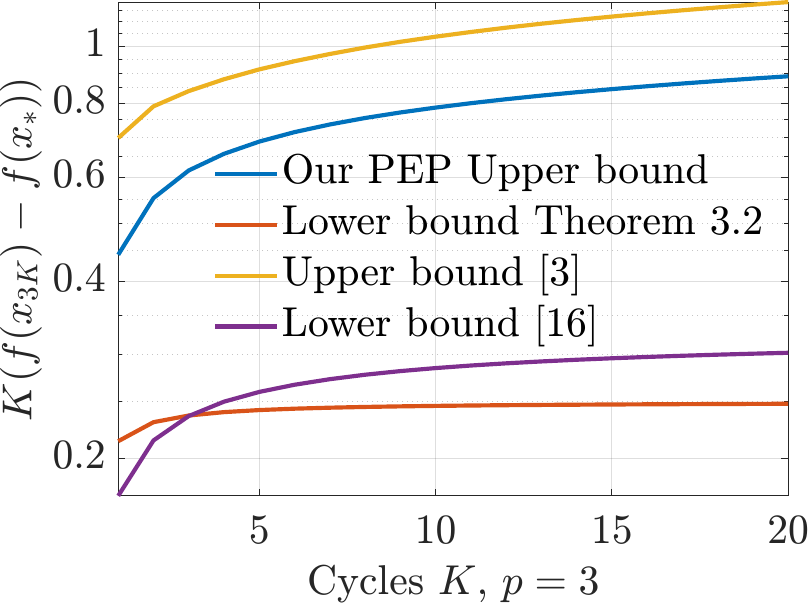}
        \caption{$3$-block~\hyperref[alg:CCD]{(CCD)} bounds}
        \label{fig:CCD3_init}
    \end{subfigure}

    \vspace{0.8em}

    % --- Row 2: INIT fits ---
    \begin{subfigure}[b]{0.45\textwidth}
        \centering
        \includegraphics[width=\textwidth]{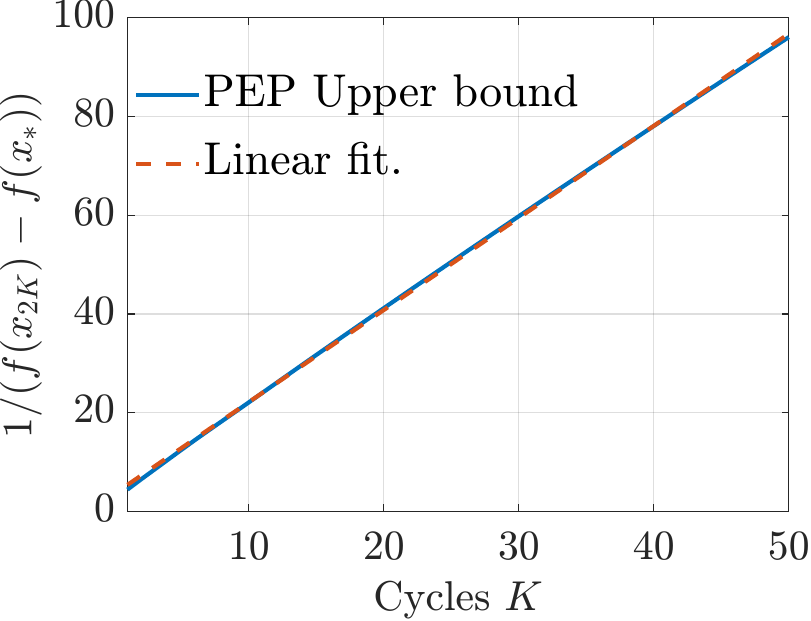}
        \caption{$2$-block~\hyperref[alg:CCD]{(CCD)} numerical fit}
        \label{fig:CCD2_init_fit}
    \end{subfigure}
    \hspace{0.001\textwidth}
    \begin{subfigure}[b]{0.45\textwidth}
        \centering
        \includegraphics[width=\textwidth]{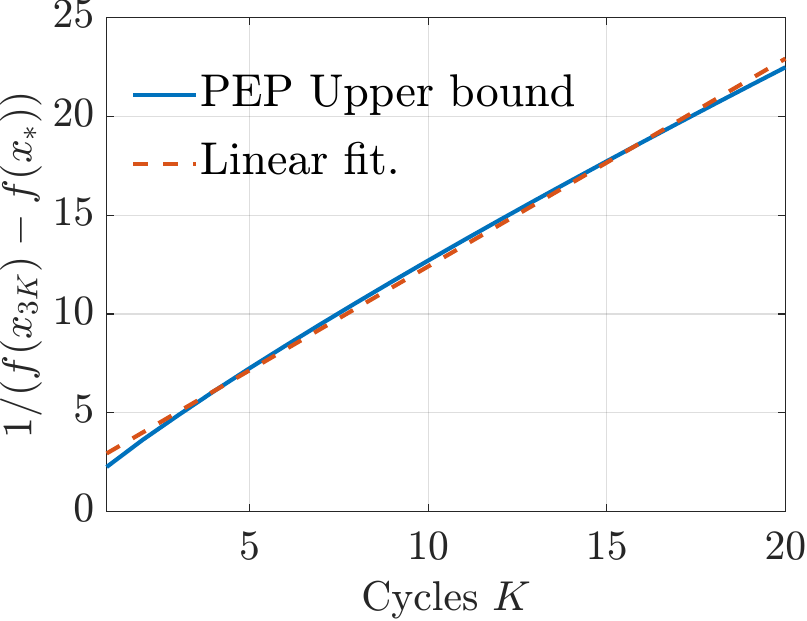}
        \caption{$3$-block~\hyperref[alg:CCD]{(CCD)} numerical fit}
        \label{fig:CCD3_init_fit}
    \end{subfigure}

    \caption{\small Worst-case performance bounds and numerical linear fits of our PEP upper bounds for 2-block and 3-block~\hyperref[alg:CCD]{(CCD)} methods in the Setting INIT.}
    \label{fig:CCD_init}
\end{figure}

\subsection{Upper bound on the convergence rate of~\hyperref[alg:CCD]{(CCD)} derived with PEP-based descent lemmas}\label{sec:num_desc}
\noindent Finding analytical expressions for the previous numerical upper bounds on the convergence of~\hyperref[alg:CCD]{(CCD)} is challenging. Therefore, we generalize the results of Section~\ref{sec:desc_lemma_2} using PEP to generate numerically optimized descent lemmas and derive semi-analytical bounds on the functional accuracy.
\begin{theorem}\label{th:coord_gen_descent}
    Given a number of blocks $p$ and a vector of nonnegative constants $\textbf{L}$, let $\{x_i\}_{i = \{1,\dots,N\}}$, $N = pK$, be the sequence of iterates generated by the $p$-block~\hyperref
    [alg:CCD]{(CCD)}$(\frac{1}{L_\ell})$ over a function $f$ that belongs to $\mathcal{F}^{\text{coord}}_{0,\textbf{L}}(\mathbb{R}^d)$. We suppose that there exists a nonnegative constant $C$ such that for all $k \in \{1,\dots,K-1\}$:
\begin{equation}\label{descent lemma}
  f(x_{kp}) -  f(x_{(k+1)p}) \geqslant C \|\nabla f(x_{kp})\|^2,
\end{equation}
then for all $k \in \{1,\dots,K\}$
\begin{equation}\label{eq:coord_bound_desc_lemma}
    f(x_{kp})-f(x_*) \leqslant \frac{1}{C}  \frac{R^2}{k+m},
\end{equation}
with $R = \min_{x_* \in X_*} \max_{k \in \{1,\dots,K\}} \{\|x-x_*\| : f(x) \leqslant f(x_0)\}$  and $m = \frac{2}{pL_{\text{max}}C}$ where $L_{\text{max}} = \max_{\ell \in \{1,\dots,p\}} L_{\ell}$.
\end{theorem}
\begin{proof}
As explained previously (when presenting Setting ALL) this is adapted from \cite[Corollary 3.7]{beck2013CCD}. Their proof uses the inequalities:
\begin{equation*}
   \forall k \in 1,\dots, K, \; \|x_{pk}-x_{*}\| \leqslant \max_{x_* \in X_*} \max_{x \in \mathbb{R}^{d}} \{\|x-x_*\|:\; f(x) \leqslant f(x_0\}.
\end{equation*}
But remains valid, considering the optimal point $x_*$ such that
\begin{equation*}
  \forall k \in 1,\dots, K, \;  \|x_{pk}-x_*\| \leqslant \min_{x_* \in X_*} \max_{k \in 1,\dots,K} \|x_{pk}-x_{*}\|,
\end{equation*}
and using these inequalities instead. Furthermore, the proof in \cite[Corollary 3.7]{beck2013CCD} is based on the following descent lemma
\begin{equation}\label{beck lemma}
    f(x_{kp}) -  f(x_{(k+1)p})  \geqslant \frac{1}{4L_{\max}(1+p^3\frac{L_{\max}^2}{L^2_{\min}})} \|\nabla f(x_{kp})\|^2 ,
\end{equation}
 which can be replaced by the descent lemma~\eqref{descent lemma} giving the desired upper bound~\eqref{eq:coord_bound_desc_lemma}.
\qed \end{proof}
\noindent The main challenge here is to find a constant \( C \) such that the descent property in~\eqref{descent lemma} is satisfied. Unlike the simpler case of only two blocks of coordinates for which a valid constant $C$ is given in~\eqref{desc_lemma}, it is difficult to analytically determine a better constant than the one used in descent lemma~\eqref{beck lemma}. However, we can numerically obtain an optimized constant by solving the following Performance Estimation Problem:
\begin{equation}\tag{PEP-coord-descent}\label{PEP-coord-descent}
\begin{aligned}
    C_{\text{opt}} = &\min_{\mathcal{S}_p=\{(x_i,g_i,f_i)\}_{i \in \{1,\dots p\}}} f(x_0) - f(x_p),\\
     & \text{such that } \mathcal{S}_p \text{ is } \mathcal{F}^{\text{coord}}_{0,\mathbf{L}}(\mathbb{R}^d)\text{-interpolable}, \\
     & x_1, \ldots, x_p \text{ are generated from } x_0 \text{ by}~\hyperref[alg:CCD]{(CCD)}, \\
    & \|\nabla f(x_0)\|^2 = 1.
\end{aligned}
\end{equation}
By definition, it is straightforward to see that \( C_{\text{opt}} \) is the optimal constant for which~\eqref{descent lemma} remains valid. Since~\hyperref[alg:CCD]{(CCD)} is a fixed-step first-order BCD algorithm, Problem~\eqref{PEP-coord-descent} can also be relaxed into a convex SDP program, following a similar construction to the one presented for Problem~\ref{f-PEP-coord}. This provides a computationally efficient way to compute optimized constants ensuring the validity of~\eqref{descent lemma}, which, by Theorem~\ref{th:coord_gen_descent}, leads to a semi-analytical upper bound on the convergence of~\hyperref[alg:CCD]{(CCD)}.
\begin{remark}
Note that the approach for automatically designing a descent lemma for~\hyperref[alg:CCD]{(CCD)} can be directly adapted to any deterministic fixed-step first-order BCD method for which a convex PEP formulation exists. Furthermore, computing optimized descent lemmas is computationally more efficient, since establishing a descent lemma with PEP only requires solving an SDP involving \( p \) matrices of size \( (p+2) \times (p+2) \), where \( p \) is the number of blocks. 
\end{remark}
We present numerical results using this descent lemma-based approach in Figure~\ref{fig:desc_CCD} for the $2$-block~\hyperref[alg:CCD]{(CCD)}$\left(\frac{1}{L_\ell}\right)$ algorithm under the assumptions of Setting ALL with $\textbf{L} = (1,1)$ and $R_a = 1$.
\begin{figure}[H]
    \centering
        \includegraphics[width=0.5\textwidth]{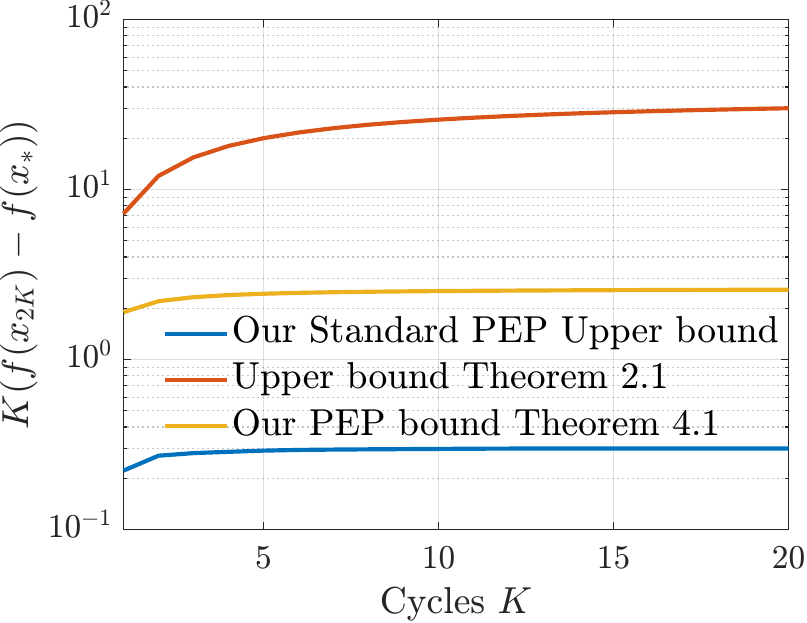}
    \caption{\small Comparison between the standard PEP bound (blue), the upper bound using PEP to establish a descent lemma (yellow) and the best known analytical bound (red) on the convergence of $2$-block~\hyperref[alg:CCD]{(CCD)}  with $\textbf{L} = (1,1)$.}
    \label{fig:desc_CCD}
\end{figure}
\noindent The optimal value of Problem~\eqref{PEP-coord-descent} in this case is $C_{\text{opt}} = 0.38$. Figure~\ref{fig:desc_CCD} shows that the upper bound, obtained using our PEP-based descent lemma, outperforms the bound given in Theorem~\ref{th:coord_beck_ccd}, originally derived in~\cite{beck2013CCD} using the suboptimal descent Lemma~\ref{beck lemma}. Compared to the standard PEP which guarantees the optimality of the upper bound with respect to the set of interpolation constraints used, the descent lemma approach yields more conservative upper bounds on the worst-case convergence of~\hyperref[alg:CCD]{(CCD)}. However, as illustrated here, the approach based on PEP-generated descent lemma provides a suitable compromise between the tightness and simplicity of the resulting upper bound.

\subsection{Linear dependency of the convergence rate for~\hyperref[alg:CCD]{(CCD)} with respect to the number of blocks $p$ }\label{sec:blocks}
In Figure~\ref{fig:p_dep}, we plot the upper bound on the function accuracy, $f(x_{pK}) - f(x_*)$ obtained using PEP, with respect to the number of blocks $p$ for $1$, $2$, and $3$ cycles. We provide a linear fit for these plots. The numerical results of Figures~~\ref{fig:p_dep}  suggest that our upper bound grows linearly with the number of blocks $p$. This result represents an improvement over the analytical bound in Theorem~\ref{th:coord_beck_ccd}, which exhibits a cubic dependence on $p$ under the stronger assumptions of Setting ALL. Given the lower bound on the convergence of~\hyperref[alg:CCD]{(CCD)}$\left(\frac{1}{L_\ell}\right)$ derived in Theorem~\ref{thm:coord_lower_bound}, these numerical results suggest that, under Setting INIT, a linear dependence of the worst-case convergence rate of~\hyperref[alg:CCD]{(CCD)}$\left(\frac{1}{L_\ell}\right)$  on the number of blocks $p$ is optimal in the sense that any valid upper bound on the worst-case of~\hyperref[alg:CCD]{(CCD)} has to at least grow linearly with the number of blocks $p$.
\vspace{-5mm}
\begin{figure}[H]
    \centering

    % --- Row 1: Full-width plot ---
    \begin{subfigure}[b]{0.5\textwidth}
        \centering
        \includegraphics[width=\textwidth]{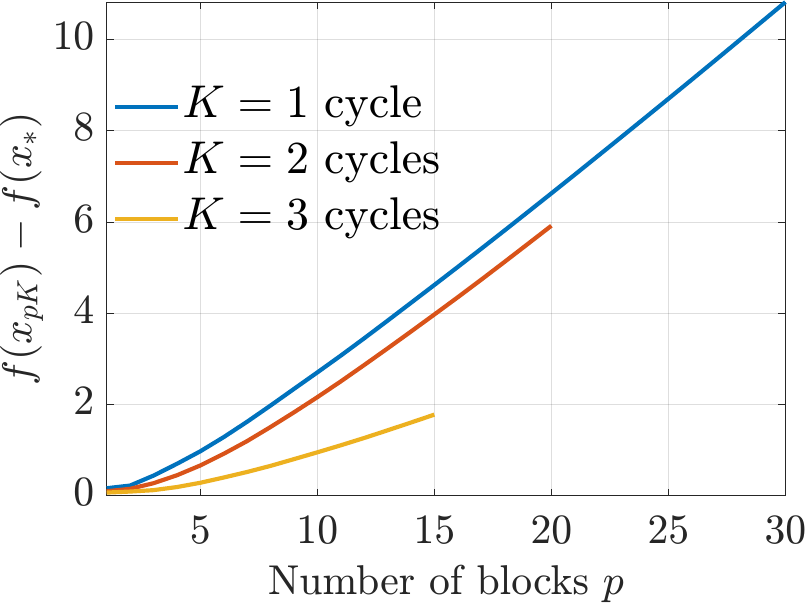}
        \caption{Evolution of our PEP upper bound with respect to the number of blocks $p$ in Setting INIT.}
        \label{fig:p_dep}
    \end{subfigure}

    \vspace{1em}

    % --- Row 2: Subplots for 1/2/3 cycles ---
    \begin{subfigure}[b]{0.4\textwidth}
        \centering
        \includegraphics[width=\textwidth]{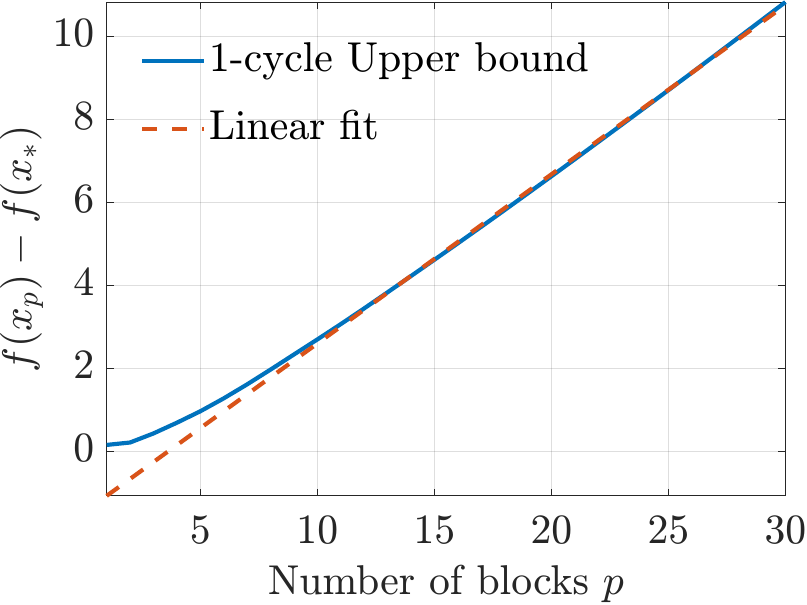}
        \caption{$1$ cycle}
        \label{fig:blocks_plot1}
    \end{subfigure}
    \hspace{0.001\textwidth}
    \begin{subfigure}[b]{0.4\textwidth}
        \centering
        \includegraphics[width=\textwidth]{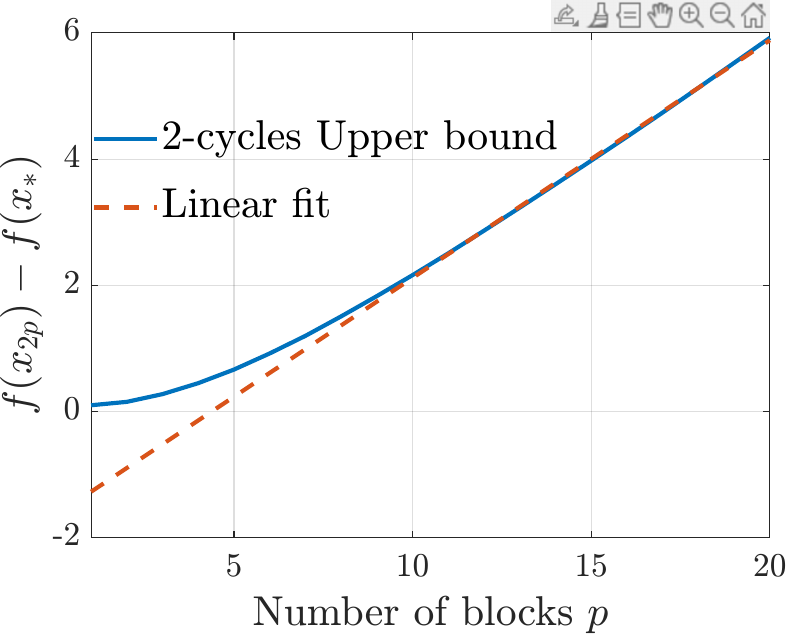}
        \caption{$2$ cycles}
        \label{fig:blocks_plot2}
    \end{subfigure}
    \hspace{0.001\textwidth}
    \begin{subfigure}[b]{0.4\textwidth}
        \centering
        \includegraphics[width=\textwidth]{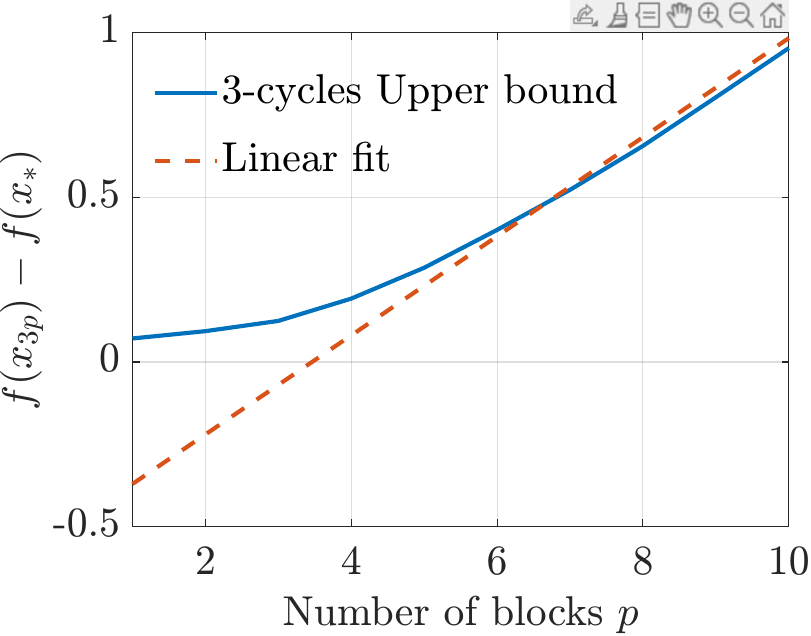}
        \caption{$3$ cycles}
        \label{fig:blocks_plot3}
    \end{subfigure}

    \caption{\small Linear behavior of our PEP upper bounds with respect to the number of blocks $p$ for~\hyperref[alg:CCD]{(CCD)} in the Setting INIT. The top plot shows the evolution of the PEP upper bound on worst-case of~\hyperref[alg:CCD]{(CCD)} with the number of blocks for $1$, $2$ and $3$ cycles while the bottom row shows numerical linear fits for each of these plots.}
    \label{fig:p_dep}
\end{figure}

\subsection{Optimal constant step sizes for~\hyperref[alg:CCD]{(CCD)}}\label{sec:opt_steps}

\noindent We now utilize PEP to gain insights into the optimal choice of step sizes for~\hyperref[alg:CCD]{(CCD)}$(\gamma_\ell)$. We specifically focus on step sizes of the form $\gamma_\ell = \frac{\gamma}{L_\ell}$ for each coordinate block $\ell \in \{1,\dots,p\}$. Under Setting INIT, we can assume without loss of generality that $\textbf{L} = (1,\dots,1)$ and $R_i = 1$, thanks to the scaling law provided in Theorem~\ref{th:coord_L_invar}. Figure~\ref{fig:coords_opt_steps} illustrates the evolution of the PEP upper bound as a function of the step size $\gamma$ for the $2$-block, $3$-block, and $4$-block variants of~\hyperref[alg:CCD]{(CCD)}$\left(\frac{\gamma}{L_\ell}\right)$ after one cycle ($K = 1$) and three cycles ($K = 3$). Additionally, Table~\ref{tab:opt_steps} provides approximate numerical values of the optimal step sizes minimizing the PEP upper bound for these cases. We observe that the optimal step size decreases as either the number of blocks or the number of cycles increases. This observation may be explained by the fact that increasing the number of coordinate blocks $p$ or cycles $K$ leads to a higher total number of partial gradient updates, thereby making shorter step sizes more appropriate. Furthermore, we note that the optimal step sizes identified through our numerical experiments are consistently smaller ($\gamma_\ell \leq 1$) than the conjectured optimal step sizes for full gradient descent discussed in~\cite{drori2014perf,taylor2017smooth}. Unfortunately, deriving analytical expressions for these optimal step sizes currently appears out of reach.
\begin{figure}[H]
    \centering
    \begin{subfigure}[b]{0.45\textwidth}
        \centering
        \includegraphics[width=\textwidth]{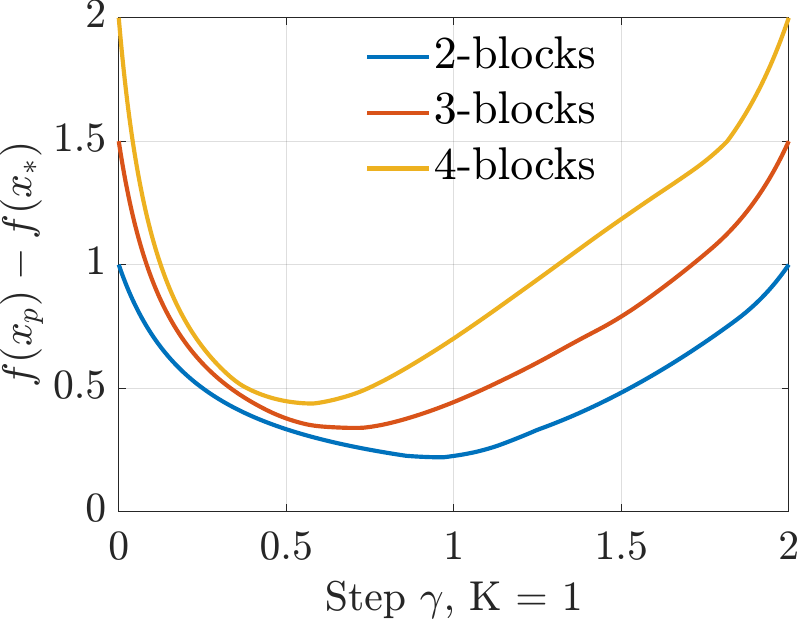}
        \caption{After $1$ cycle}
        \label{fig:opt_steps_1}
    \end{subfigure}
    \hspace{0.001\textwidth}
    \begin{subfigure}[b]{0.45\textwidth}
        \centering
        \includegraphics[width=\textwidth]{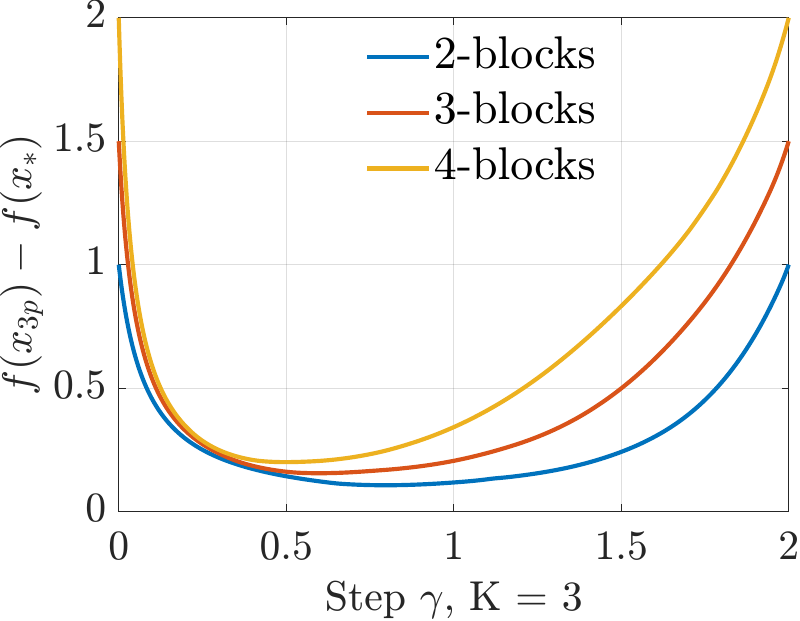}
        \caption{After $3$ cycles}
        \label{fig:opt_steps_3}
    \end{subfigure}
    \caption{\small PEP upper-bound on the convergence of~\hyperref[alg:CCD]{CCD}\,$(\frac{\gamma}{L_\ell})$ in the Setting INIT with respect to the step-size $\gamma$ for $2$, $3$, and $4$ blocks of coordinates after $1$ and $3$ cycles.}
    \label{fig:coords_opt_steps}
\end{figure}
\begin{table}[H]
    \centering
    \resizebox{1\textwidth}{!}{%
    \begin{tabular}{|c||c||c|}
        \hline
        Number of blocks $p$ & Optimal step size $(K = 1)$ & Optimal step size $(K = 3)$ \\
        \hline
        2  & 0.967 & 0.796\\ 
        \hline
        3  & 0.700 & 0.596\\
        \hline
        4  & 0.576 & 0.496 \\   
        \hline
    \end{tabular}%
    }
    \caption{\small Optimal choice of step size $\gamma$ for $1$ and $3$ cycles of $2$-block, $3$-block and $4$-block~\hyperref[alg:CCD]{(CCD)}$\left(\frac{\gamma}{L_\ell}\right)$ }
    \label{tab:opt_steps}
\end{table}

\subsection{Improved upper bound on the worst-case convergence rate of~\hyperref[alg:AM]{(AM)} in Setting ALL.}\label{sec:AM}

\noindent We now turn our attention to the alternating minimization algorithm which alternatively minimize exactly the objective function along a block of coordinates:
\begin{algorithm}[H]
    \caption{Alternating minimization (AM)}\label{alg:AM}
\begin{algorithmic}
\State \textbf{Input} function $f$ defined over $\mathbb{R}^d$ with $p$ blocks, starting point $x_0 \in \mathbb{R}^d$, number of cycles $K$.
\State Define $N = pK$. For $i = 1 \dots N$ ,
\State \hspace{1.5cm} Set $\ell = \text{mod}(n,p) + 1$
\State \hspace{1.5cm} $x_i = \arg \min_{z = x_{i-1} + U_\ell\Delta x^{(\ell)}, \; \Delta x^{(\ell)} \in \mathbb{R}^{d_\ell}} f(z)$
\end{algorithmic}
\end{algorithm}
\noindent As we did previously for~\hyperref[alg:CCD]{(CCD)}, we present a slightly adapted convergence theorem for~\hyperref[alg:AM]{(AM)} so that the bound remains valid under the assumptions of our Setting ALL.
\begin{theorem}\cite[Theorem 5.2]{beck2013CCD}\label{th:coord_am_bound} Given a vector of nonnegative constants $\textbf{L} = (L_1,L_2)$ and a function $f \in \mathcal{F}^{\text{coord}}_{0,\textbf{L}}(\mathbb{R}^d)$, let $\{x_i\}_{i \in \{1,\dots,N\}}$, $N = 2K$, then the sequence of iterates generated by the $K$ cycles of $2$-block~\hyperref[alg:AM]{(AM)} on $f$ satisfies:
\begin{equation*}
    f(x_{2K}) - f_* \leq \frac{2 \min \{L_1,L_2\} R^2_a}{K-1},
\end{equation*}
where $f_*$ is the minimal value of $f$ and $R_a = \min_{x_* \in X_*} \max_{k \in 1,\dots,K} \|x_{pK}-x_*\|$.
\end{theorem}
\begin{proof}
 Same reasoning as in Theorem~\ref{th:ccd_beck} for~\hyperref[alg:CCD]{CCD} applies. \qed \end{proof}
\hyperref[alg:AM]{(AM)} is not a fixed-step first-order algorithm and thus we need to adapt the Gram lifting approach proposed in Section~\ref{sec:cvx_PEP} for fixed-step algorithms to reformulate Problem~\eqref{R-PEP-coord} as an SDP for~\hyperref[alg:AM]{(AM)}. This requires considering the matrices $P^{(\ell)} = [g^{(\ell)}_0, \dots,g^{(\ell)}_N,\allowbreak x^{(\ell)}_0, \dots,x^{(\ell)}_N]$ and the corresponding Gram matrices $G^{(\ell)} = P^{(\ell)\top} P^{(\ell)} \in \mathbb{S}^{2N+2}$. Then, for all \( \ell \in \{1,\dots,p\} \) and \( i \in \{0,\dots,N\} \), we have $x^{(\ell)}_i = P^{(\ell)} h_i$, with $h_i = e_{N+2+i}$, and $ g^{(\ell)}_i = P^{(\ell)} u_i$, with $u_i = e_{i+1}$. \( \{e_i\}_{i \in \{1,\dots,2N+2\}} \) is the canonical basis of \( \mathbb{R}^{2N+2} \). The update step in~\hyperref[alg:AM]{(AM)}, given by $    x_i = \text{argmin}_{z = x_{i-1} + U_\ell\Delta x^{(\ell)}, \; \Delta x^{(\ell)} \in \mathbb{R}^{d_\ell}} f(z)$, can be modeled in a PEP through the conditions $ \|x^{(s)}_{i}\|^2 = \|x^{(s)}_{i-1}\|^2, \quad \forall s \in \{1,\dots,p\}, \; s \neq \ell$ and $\|g^{(\ell)}_i\|^2 = 0$, which can be expressed as linear constraints on the Gram matrices $G^{(\ell)}$. The interpolation conditions, performance criterion, and initial condition can be handle in the same manner as for fixed-step first-order algorithms.

\medskip

\noindent We analyze the $2$-block alternating minimization algorithm~\hyperref[alg:AM]{(AM)} under Setting ALL for various choices of Lipschitz constants with $R_a = 1$. Although our framework accommodates an arbitrary number of blocks, we restrict our analysis to the $2$-block case to facilitate comparison between our numerical upper bounds and the analytical bound previously established in~\cite{beck2013CCD} (see Theorem~\ref{th:coord_am_bound}). Figure~\ref{fig:coords_AM} indicates that the existing bound for the $2$-block~\hyperref[alg:AM]{(AM)} can be improved by approximately a multiplicative factor of $2$. This numerical result confirms the observations regarding the convergence of~\hyperref[alg:AM]{(AM)} made in~\cite{kamri} using a PEP framework over the class of smooth convex functions.

\begin{figure}[H]
    \centering
    \begin{subfigure}[b]{0.4\textwidth}
        \centering
        \includegraphics[width=\textwidth]{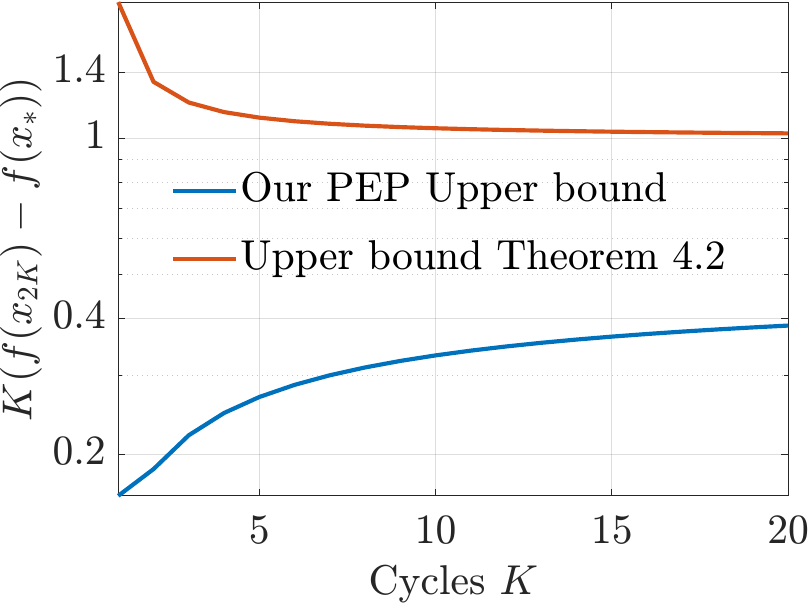}
        \caption{$\mathbf{L} = (1,1)$}
        \label{fig:AM11}
    \end{subfigure}
    \hspace{0.001\textwidth}
    \begin{subfigure}[b]{0.4\textwidth}
        \centering
        \includegraphics[width=\textwidth]{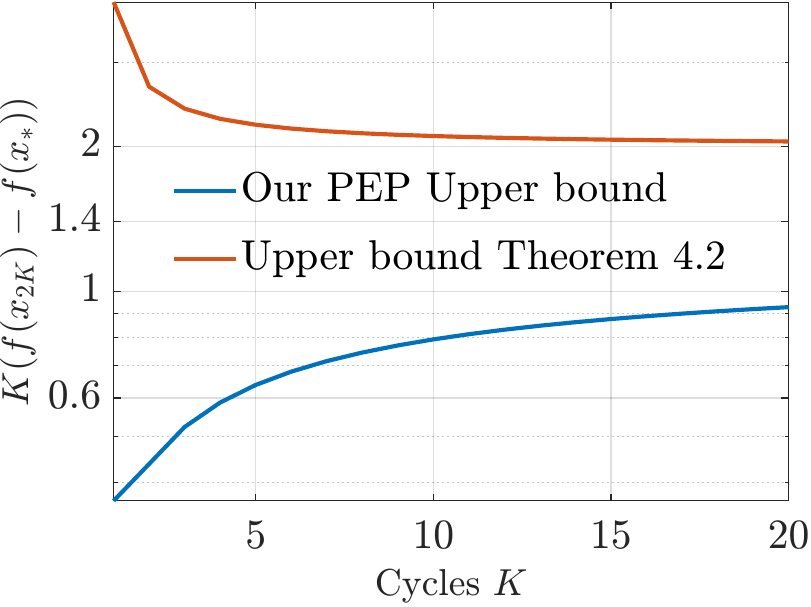}
        \caption{$\mathbf{L} = (2,3)$}
        \label{fig:AM23}
    \end{subfigure}
    \hspace{0.001\textwidth}
    \begin{subfigure}[b]{0.4\textwidth}
        \centering
        \includegraphics[width=\textwidth]{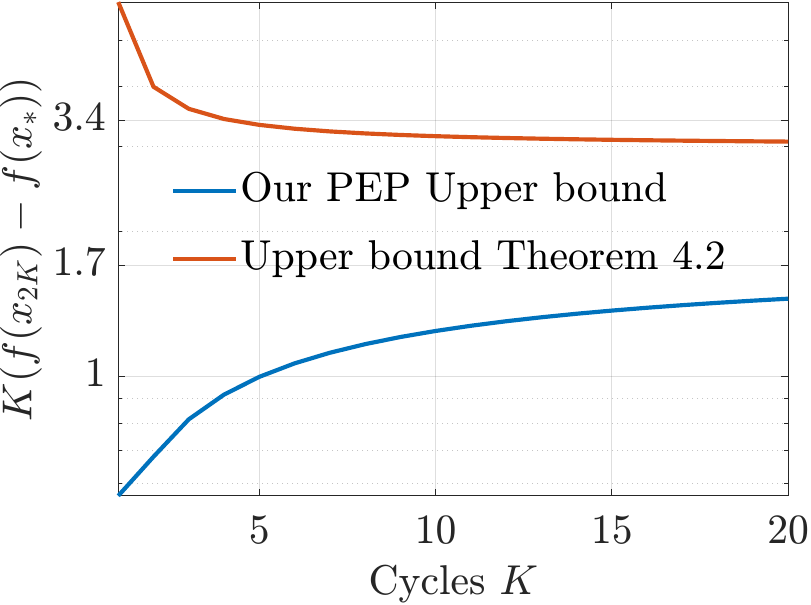}
        \caption{$\mathbf{L} = (3,5)$}
        \label{fig:AM35}
    \end{subfigure}

    \caption{Comparison of upper bounds on the worst-case for 2-block~\hyperref[alg:AM]{(AM)} for $\mathbf{L} = (1,1)$, $\mathbf{L} = (2,3)$, and $\mathbf{L} = (3,5)$ in the Setting ALL.}
    \label{fig:coords_AM}
\end{figure}

\subsection{\hyperref[alg:CACD]{(CACD)} has a slower convergence rate than~\hyperref[alg:RACD]{(RACD)}}\label{sec:CACD}

We now introduce a third less standard deterministic fixed-step first-order algorithm which is a cyclic version of the accelerated random coordinate descent algorithm from \cite{fercoq2015coord}. We denote this algorithm by~\hyperref[alg:CACD]{(CACD)}. Our interest with~\hyperref[alg:CACD]{(CACD)} lies in comparing its worst-case performance to that of its random version~\hyperref[alg:RACD]{(RACD)}, which features a proven accelerated convergence rate on average.
\begin{algorithm}[H]
\caption{Cyclic Accelerated Coordinate Descent (CACD) }\label{alg:CACD}
\begin{algorithmic}
\State \textbf{Input}  function $f$ defined over $\mathbb{R}^d$ with $p$ blocks, starting points $x_0 = z_0 \in \mathbb{R}^d$, $\theta_{0} = \frac{1}{p}$, number of cycles $K$, step-size $\{\gamma_{\ell}\}^p_{i = 1}$
\State Define $N = pK$. For $i = 1 \dots N$ 
\State \hspace{1.5cm} Set $\ell = \text{mod}(i,p) + 1$
\State \hspace{1.5cm} $y_{i-1} = (1-\theta_{i-1})x_{i-1} + \theta_{i-1} z_{i-1}$
\State \hspace{1.5cm} $z_i = z_{i-1} - \frac{\gamma_{\ell}}{p \theta_{i-1}} U_\ell \nabla^{(\ell)}f(y_{i-1})$
\State \hspace{1.5cm} $x_i = y_{i-1} + p \theta_{i-1} (z_{i} - z_{i-1})$
\State \hspace{1.5cm} $\theta_{i} = \frac{\sqrt{\theta_{i-1}^4 + 4 \theta_{i-1}^2} -\theta_{i-1}^2}{2}$
\end{algorithmic}
\end{algorithm}

\begin{algorithm}[H]
\caption{Random Accelerated Coordinate Descent (RACD) }\label{alg:RACD}
\begin{algorithmic}
\State \textbf{Input}  function $f$ defined over $\mathbb{R}^d$ with $p$ blocks, starting points $x_0 = z_0 \in \mathbb{R}^d$, $\theta_{0} = \frac{1}{p}$, number of steps $N$, step-size $\{\gamma_{\ell}\}^p_{i = 1}$
\State For $i = 1 \dots N$ 
\State \hspace{1.5cm} Sample uniformly at random $\ell \in \{1,\dots,p\}$
\State \hspace{1.5cm} $y_{i-1} = (1-\theta_{i-1})x_{i-1} + \theta_{i-1} z_{i-1}$
\State \hspace{1.5cm} $z_i = z_{i-1} - \frac{\gamma_{\ell}}{p \theta_{i-1}} U_\ell \nabla^{(\ell)}f(y_{i-1})$
\State \hspace{1.5cm} $x_i = y_{i-1} + p \theta_{i-1} (z_{i} - z_{i-1})$
\State \hspace{1.5cm} $\theta_{i} = \frac{\sqrt{\theta_{i-1}^4 + 4 \theta_{i-1}^2} -\theta_{i-1}^2}{2}$
\end{algorithmic}
\end{algorithm}
\begin{theorem}\cite[Theorem 3]{fercoq2015coord}
Given a number of blocks $p$, a vector of nonnegative constants $\textbf{L} = (L_1, \dots, L_p)$ and a function $f \in \mathcal{F}^{\text{coord}}_{0,\textbf{L}}(\mathbb{R}^d)$, let $\{x_i\}_{i \in \{1,\dots,N\}}$ be the iterates generated by $N$ steps of $p$-block of random accelerated coordinate descent with step-sizes $\gamma_\ell = \frac{1}{L_\ell}, \; \forall \ell \in \{1,\dots,p\}$ on $f$, then it holds that:
\begin{equation*}
    \mathbb{E}[f(x_N) - f^*] \leqslant \frac{4p^2}{(N-1 + 2p)^2} R^2
\end{equation*}
where $f^*$ is the minimal value of $f$ and $R^2 = \left(1 - \frac{1}{p}\right) \left(f(x_0) - f^*\right) + \frac{1}{2} \|x_{0} - x_*\|_{\textbf{L}}^2$.
\end{theorem}
\begin{proof}
With the notations of Theorem $3$ in~\cite{fercoq2015coord}, taking $\tau = 1$, it is straightforward to show that if $f$ belongs to $\mathcal{F}^{\text{coord}}_{0,\textbf{L}}(\mathbb{R}^d)$, $f$ verifies the $\text{(ESO)}$ condition~\cite[Assumption 1]{fercoq2015coord} with $v_\ell = L_\ell, \; \forall \ell \in \{1,\dots,p\}$, thus the proof of Theorem $3$ in~\cite{fercoq2015coord} remains valid in our case and we obtain the desired result.
\qed \end{proof}
Note that~\hyperref[alg:CACD]{(CACD)} is a deterministic fixed-step first-order BCD algorithm and convex relaxation of the PEP for this algorithm can be derived using the Gram matrix lifting technique presented in Section~\ref{sec:cvx_PEP}. Since~\hyperref[alg:RACD]{(RACD)} is a random algorithm its analysis using PEP requires some adjustment to the framework we developed for deterministic algorithms.
If \( p \) is the number of blocks and \( N \) the number of steps, then a random algorithm can generate up to \( p^N \) different sequences of block of coordinates. Each sequence corresponds to a deterministic algorithm, denoted by \( \mathcal{M}^{\text{coord}}_r \) for $r \in \{1,\dots,p^N\}$, which can be represented and analyzed using PEP with a corresponding set \( \mathcal{S}^r_N = \{(x_{i,r}, g_{i,r}, f_{i,r})\}_{i \in I} \) with $I = \{0,1\dots,N,*\}$, as described earlier. To represent the average performance of the random algorithm under analysis, we need to impose that all the corresponding deterministic algorithms are applied on the same function. This is done by defining the set $\mathcal{S}_N = \bigcup^{p^N}_{r = 1} \mathcal{S}^r_N$ and imposing the interpolation conditions on $\mathcal{S}_N$ to ensure that all the sets $\mathcal{S}^r_N$ are interpolable by the same function. Furthermore, we also need to impose the same initial iterate for all the algorithms and measure the performance with respect to the same optimal point. This is done in PEP by imposing the conditions $x_{0,r} = x_0$ and $\{x_{*,r}, g_{*,r}, f_{*,r}\} = \{0, 0, 0\}$ for all $r \in \{1,\dots,p^N\}$. We then choose the expectation of the performance of each deterministic algorithm \( \mathcal{M}^{\text{coord}}_r \) as the performance criterion. Consequently, the PEP for a random algorithm can be conceptually formulated as:

\begin{equation}\tag{Rdn-PEP}\label{rdn-PEP}
\begin{aligned}
    &\sup_{\mathcal{S}^r_N =  \{(x_{i,r}, g_{i,r}, f_{i,r})\}_{i \in I}} \sum^{p^N}_{r = 1} \mathbb{P}_r \mathcal{P}(\mathcal{S}^r_N),\\
     & \text{such that } \mathcal{S}_N = \bigcup^{p^N}_{r = 1} \mathcal{S}^r_N\\
     &\mathcal{S}_N \text{ is } \mathcal{F}^{\text{coord}}_{0,\mathbf{L}}(\mathbb{R}^d)\text{–interpolable}, \\
     & x_{1,r}, \ldots, x_{N,r} \text{ by method } \mathcal{M}_r^{\text{coord}} \quad \forall r \in \{1,\dots,p^N\}, \\
     & x_{0,r} = x_0 \quad \forall r \in \{1,\dots,p^N\}\\
    & \{x_{*,r}, g_{*,r}, f_{*,r}\} = \{0, 0, 0\}, \quad \forall r \in \{1,\dots,p^N\} \\
    & \|x_{0} - x_{*}\|_2 \leq R, 
\end{aligned}
\end{equation}
\noindent where \( \mathbb{P}_r \) represents the probability that the sequence \( r \) of blocks occurs. For fixed-step first-order random BCD algorithms, this problem can be cast as a convex semidefinite program using a similar construction to the one presented above for fixed-step first-order deterministic algorithms. However, the resulting SDP involves \( p \) matrices of size \( (Np^N+2) \times (Np^N+2) \), making it computationally prohibitive for analyzing random algorithms when the number of blocks and steps are large.

\medskip

\noindent We now present numerical experiments supporting the observation that~\hyperref[alg:CACD]{(CACD)} with step sizes $\gamma_\ell = \frac{1}{L_{\ell}}$, for all $\ell \in \{1,\dots,p\}$, exhibits a slower convergence rate compared to its randomized variant. We place ourselves in Setting INIT. As previously discussed, this allows us to assume without loss of generality that $\textbf{L} = (1,\dots,1)$ and $R_i = 1$ due to the scaling law provided by Theorem~\ref{th:coord_L_invar}. In Figure~\ref{fig:7}, we present valid upper bounds on the worst-case performance of $2$-block and $3$-block~\hyperref[alg:CACD]{(CACD)}$\left(\frac{1}{L_\ell}\right)$, multiplied by the squared number of cycles $K^2$. This quantity appears to be increasing in $K$, suggesting a convergence rate slower than $\mathcal{O}(1/K^2)$. To further investigate this gap between~\hyperref[alg:CACD]{(CACD)} and its random variant~\hyperref[alg:RACD]{(RACD)}, we report in Table~\ref{tab:coords_CACD} PEP upper bounds for all possible sequences of coordinate-block updates after $K = 2$ cycles of $2$-block~\hyperref[alg:CACD]{(CACD)}$\left(\frac{1}{L_\ell}\right)$. Additionally, we provide an upper bound on the worst-case expected performance after $N = 4$ steps of~\hyperref[alg:RACD]{(RACD)}$\left(\frac{1}{L_\ell}\right)$, denoted by $\mathcal{W}^{RACD}$. The bound $\mathcal{W}^{RACD}$ is computed using the PEP framework for random algorithms described in Problem~\eqref{rdn-PEP}. It is noteworthy that $\mathcal{W}^{RACD}$ is smaller than all the upper bounds for every possible deterministic sequence of updates. This highlights the fact that randomness seems to play a significant role in acceleration for BCD algorithms over the class $\mathcal{F}^{\text{coord}}_{0,\textbf{L}}(\mathbb{R}^d)$. Intuitively, it appears that for each fixed sequence of coordinate block updates, one can construct a specific worst-case function on which~\hyperref[alg:CACD]{(CACD)}$\left(\frac{1}{L_\ell}\right)$ performs poorly. However, it seems that no single function leads to poor performance across all update sequences and thus on average for~\hyperref[alg:RACD]{(RACD)}$\left(\frac{1}{L_\ell}\right)$ .

\begin{table}[H]
    \centering
    \resizebox{1\textwidth}{!}{%
    \begin{tabular}{|c|c||c|}
        \hline
        Block choice type & Ordered block choices  & Worst-case \\ 
        \hline
         Cyclic (= $\mathcal{W}^\text{CACD}$) & 1    2    1    2 \text{ or }  2    1    2    1& 0.14429  \\
        \hline
         Fixed choice  & 1    2    2    1  \text{ or }  2    1    1    2   & 0.14988 \\
        \hline
        Fixed choice &       1    2    1    1 \text{ or } 2    1    2    2 &  0.16453 \\
        \hline
        Fixed choice & 1    1    2    1 \text{ or }  2    2    1    2 &  0.19574\\
        \hline
        Fixed choice &  1    2    2    2 \text{ or }  2    1    1    1 & 0.19905 \\
        \hline
        Fixed choice &  1    1    2    2   \text{ or } 2    2    1    1& 0.23462 \\
        \hline
        Fixed choice & 1    1    1    2 \text{ or } 2    2    2    1  & 0.25517 \\
        \hline
        Fixed choice & 1    1    1   1 \text{ or } 2    2    2    2    & 0.500\\ 
        \hline\hline
         Random (= $\mathcal{W}^\text{RACD}$)  &  at each step: pick 1 or 2 &    0.1046 \\   

        \hline
\end{tabular}%
}
\caption{Upper bounds on the convergence rate after $4$ iterations of~\hyperref[alg:CACD]{(CACD)}$\left(\frac{1}{L_\ell}\right)$ for every possible sequence of block coordinate updates and upper bound on the worst-case expected performance of~\hyperref[alg:RACD]{(RACD)}$\left(\frac{1}{L_\ell}\right)$, denoted by $\mathcal{W}^{\text{RACD}}$, obtained via the PEP formulation~\eqref{rdn-PEP}.}
 \label{tab:coords_CACD}
\end{table}
\begin{figure}[H]
    \centering
    \begin{subfigure}[b]{0.45\textwidth}
        \centering
        \includegraphics[width=\textwidth]{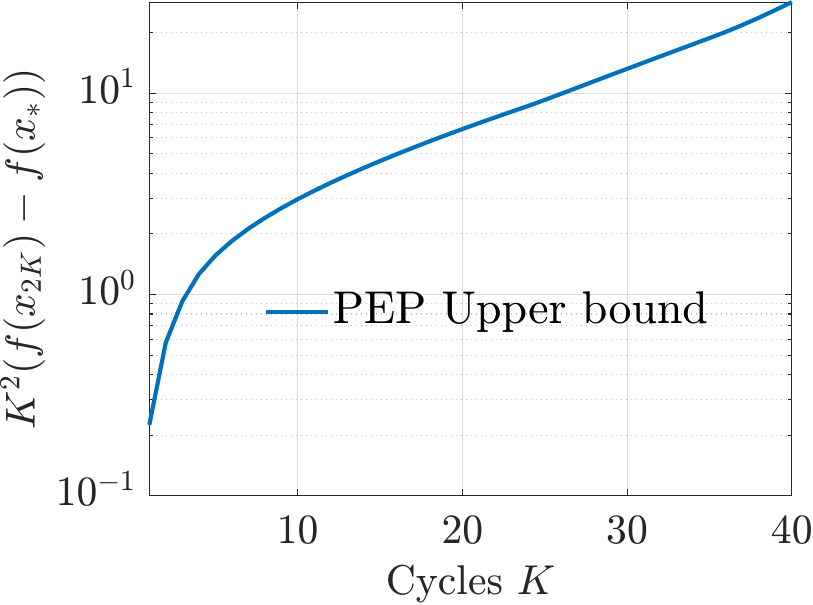}
        \caption{2-block CACD}
        \label{fig:acc_2blocks}
    \end{subfigure}
    \hspace{0.001\textwidth}
    \begin{subfigure}[b]{0.45\textwidth}
        \centering
        \includegraphics[width=\textwidth]{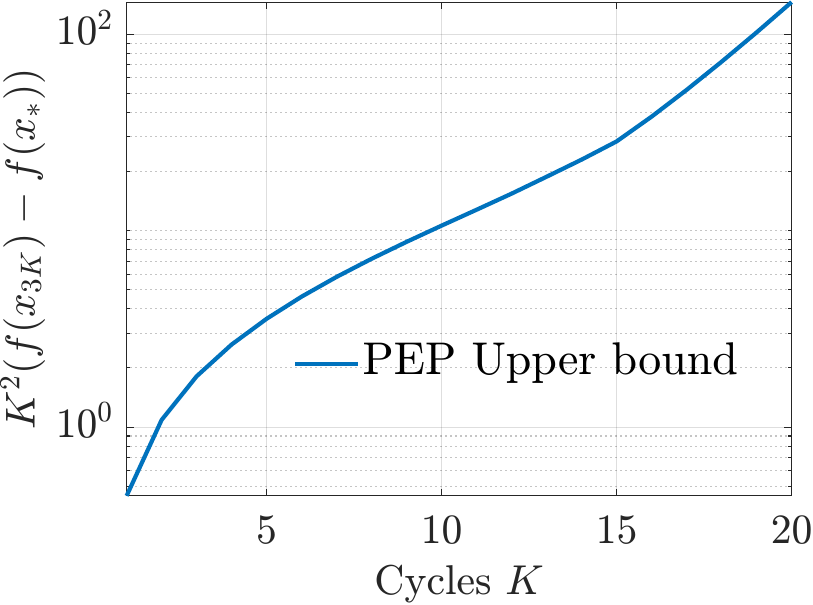}
        \caption{3-block CACD}
        \label{fig:acc_3blocks}
    \end{subfigure}

    \caption{\small PEP upper bound multiplied by the squared number of cycles $K^2$ for 2-block (left) and 3-block (right)~\hyperref[alg:CACD]{(CACD)}$\left(\frac{1}{L_\ell}\right)$, indicating convergence slower than $\mathcal{O}\left(\frac{1}{K^2}\right)$.}
    \label{fig:7}
\end{figure}

\section{Conclusion}
In this paper, we introduced a versatile convex PEP framework for the worst-case analysis of block coordinate descent algorithms. We applied this framework to provide new insights about the behaviour of several BCD algorithms. In particular, we provided improved numerical bounds for cyclic coordinate descent~\hyperref[alg:CCD]{(CCD)} and alternating minimization~\hyperref[alg:AM]{(AM)} and provided insights on the evolution of the worst-case of~\hyperref[alg:CCD]{(CCD)} with respect to the number of blocks as well as the evolution of the optimal step-size~\hyperref[alg:CCD]{(CCD)} with respect to the number of blocks and cycles. Additionally, we compared the worst-case performance of cyclic accelerated coordinate descent~\hyperref[alg:CACD]{(CACD)} to its random version and observed a slower convergence rate, indicating the significance of randomness in accelerating block coordinate descent algorithms. We demonstrated the convergence of these algorithms under the simplified constraints of Setting INIT, where only a bound on the initial distance to a minimizer is assumed, as opposed to the typically employed stronger assumptions of Setting ALL, which provide bounds on the distances between all the iterates and a minimizer. Furthermore, we established formally several results concerning the behavior of~\hyperref[alg:CCD]{(CCD)}, including scale invariance with respect to the block coordinate-wise smoothness constants, an improved descent lemma for 2-block~\hyperref[alg:CCD]{(CCD)} with a corresponding upper bound on the convergence for the residual gradient norm, and a lower bound on its worst-case performance for the function value accuracy criterion that is equal to the number of blocks multiplied by the worst-case performance of full gradient descent on smooth convex functions. We believe our PEP framework to be a useful tool to guide future research on block coordinate-wise algorithms for example by helping the tuning of step-sizes.

\medskip

\noindent Open research direction include: analytically identifying our numerical upper bounds with the corresponding proofs of convergence, deriving necessary and sufficient interpolation conditions for the functional class $\mathcal{F}^{\text{coord}}_{0,\textbf{L}}(\mathbb{R}^d)$ which would give us access to the exact worst-case performance of BCD algorithms over this class of functions, developing a computationally more efficient framework for analyzing random algorithms and designing suitable acceleration schemes for deterministic BCD algorithms.

\section*{Acknowledgments}
Y. Kamri was supported by the European Union’s MARIE SKŁODOWSKA-CURIE Actions Innovative Training Network (ITN)-ID 861137, TraDE-OPT. and by the FSR program.

\bibliographystyle{plain}  % or another style like alpha, ieeetr, etc.
\bibliography{JOTA_template/references}

\appendix
\section*{Appendix}
\renewcommand{\thesection}{A} % Optional: Label appendix sections with A, B, etc.

\subsection*{A.1 Proof of Lemma~\ref{lm:coord_lower_bound}}

The following proof is inspired from the proof of \cite[Theorem 2.34]{taylor2017thesis} where a similar result for $L$-smooth convex functions is derived. We define the subspace $S_{\neq \ell} = \{x \in \mathbb{R}^d, U_{j}^{\top}x = 0, \forall j \in {1,\dots,p},\; j\neq \ell\}$  and the extended function $p_{\ell}(x) = \frac{\|U_{\ell}^{T}x\|^2}{2} + \mathbb{I}_{S_{\neq \ell}}(x)$,  
with $\mathbb{I}_{S_{\neq \ell}}(x) = \begin{cases}
        0 & \text{if} \; x \in S_{\neq \ell} \\
        + \infty & \text{otherwise}
    \end{cases}$.

\noindent Since $f$ belongs to $\mathcal{F}^{\text{coord}}_{0,\mathbf{L}}(\mathbb{R}^d)$, it satisfies condition~\eqref{eq:coord_upper_bound}. We claim that we can rewrite this condition as follows:
\begin{equation}\label{eq:cond_p}
    f(x+\Delta x) \leqslant f(x) + \langle \nabla f(x),\Delta x \rangle + L_{\ell} p_{\ell}(\Delta x)  \;\;  \forall x \in \mathbb{R}^d, \; \forall \Delta x \in \mathbb{R}^d,
\end{equation}
which now involves the full gradient of $f$ as opposed to the partial gradient $\nabla^{(\ell)} f$, and must be satisfied for all $\Delta x \in \mathbb{R}^d$ as opposed to the $U_{\ell}h^{(\ell)}$ that are only nonzero along the block $\ell$. Indeed, if $\Delta x$ is only nonzero along the coordinates of the $\ell^{th}$ block, then 
$p_{\ell}(\Delta x) = \frac{\|U_{\ell}^{T}\Delta x\|^2}{2}$, $\langle \nabla f(x),\Delta x \rangle = \langle \nabla^{(\ell)} f(x), U_{\ell}^{T}\Delta x \rangle$ and we recover \eqref{eq:coord_upper_bound}. For all other values of $\Delta x$, $p_\ell(\Delta x)$ is infinite, and \eqref{eq:cond_p} is trivially satisfied. For any fixed $x \in \mathbb{R}^d$, let us now define
\begin{equation*}
    h_x(\Delta x) = f(x+\Delta x) - f(x) - \langle \nabla f(x),\Delta x \rangle.
\end{equation*}
Condition \eqref{eq:cond_p} becomes then $h_x(\Delta x) \leqslant L_{\ell} p_{\ell}(\Delta x), \; \forall x, \Delta x \in \mathbb{R}^d$. Note that both $h_x$ and $p_\ell$ are convex functions w.r.t. $\Delta x $. As the Fenchel conjugate (with respect to $\Delta x$) reverses the order between convex functions, this is equivalent to:
\begin{equation}\label{eq:ineq_dual_proof}
     h_x^{*}(\Delta g) \geqslant (L_{\ell} p_{\ell})^*(\Delta g), \; \forall \Delta g \in \mathbb{R}^d.
\end{equation}
First, we compute the Fenchel conjugate  $(L_{\ell} p_{\ell})^*(\Delta g)$.
For any $\Delta x \in \mathbb{R}^d$ such that $\Delta x \notin S_{\neq \ell}$, we clearly have $\langle \Delta g,\Delta x \rangle - L_{\ell} p_{\ell}(\Delta x) = - \infty$. For the case where $\Delta x \in S_{\neq \ell}$, we have that:
\begin{equation*}
\begin{aligned}
        \forall \Delta x \in S_{\neq \ell},\; \Delta g \in \mathbb{R}^d, \; T(\Delta x,\Delta g) &:= \langle \Delta g,\Delta x \rangle - L_{\ell} p_\ell(\Delta g) \\
        & = \Delta g^T U_{\ell} U_{\ell}^T \Delta x - \frac{L_{\ell}}{2} \Delta x^{T} U_{\ell} U_{\ell}^{T} \Delta x \\
        & = \Delta g^T M_{\ell} \Delta x - \frac{L_{\ell}}{2} \Delta x^{T} M_{\ell} \Delta x
\end{aligned}
\end{equation*}
with $M_{\ell} = U_{\ell} U_{\ell}^T$ is an orthogonal projection matrix i.e $M_{\ell} = M_{\ell}^{T} = M_{\ell}^2$. Computing the derivative of $T$ with respect to $\Delta x$ and setting it to zero we get that the maximum, $\Delta x^*$,  verifies $\Delta x^* = \frac{1}{L_\ell} M_\ell \Delta g \in S_{\neq \ell}$. Using this we get that $\Delta g^T M_{\ell} \Delta x^* = \frac{1}{L_{\ell}} \Delta g^T M_{\ell} \Delta g$ and
\begin{equation*}
  \frac{L_{\ell}}{2}  \Delta x^{*T} M_{\ell} \Delta x^* = \frac{L_{\ell}}{2} \Delta x^{*T} M_{\ell}^2 \Delta x^* = \frac{L_{\ell}}{2} (\Delta x^{*T} M_{\ell}^{T}) (M_{\ell} \Delta x^*) = \frac{1}{2L_{\ell}} \Delta g^T M_{\ell} \Delta g.
\end{equation*}
Putting everything together we have that:
\begin{equation*}
\begin{aligned}
\forall \Delta g \in \mathbb{R}^{d}, \; (L_{\ell} p_{\ell})^*(\Delta g) &= \sup_{\Delta x \in \mathbb{R}^{d}} \langle \Delta g,\Delta x) - L_{\ell} p_{\ell}(\Delta x) \\
&=  \sup_{\Delta x \in S_{\neq \ell}} \langle \Delta g,\Delta x) - L_{\ell} p_{\ell}(\Delta x)\\
& = \frac{1}{2L_{\ell}} \Delta g^{T} M_{\ell} \Delta g = \frac{1}{2L_{\ell}} ||U_{\ell}^T\Delta g||^2.
\end{aligned}
\end{equation*}
Let us now compute an expression of the Fenchel conjugate of $h_x$ with respect to Fenchel conjugate $f^*$ of $f$:
\begin{equation*}
    \begin{aligned}
          h^*_x(\Delta g) &=\sup_{\Delta x \in \mathbb{R}^d} \langle \Delta g,\Delta x\rangle -f(x+\Delta x) + f(x) + \langle \nabla f(x),\Delta x \rangle \\
          &= f(x) + \sup_{\Delta x \in \mathbb{R}^d} \langle \nabla f(x) + \Delta g,\Delta x\rangle  -f(x+\Delta x) \\
          & =  f(x) + \sup_{x' \in \mathbb{R}^d} \langle \nabla f(x) + \Delta g, x'-x\rangle  -f(x') \; \; \text{with}\; x' = x+\Delta x \\
          & = f(x) - \langle \nabla f(x) + \Delta g, x) + \sup_{x' \in \mathbb{R}^d} \langle \nabla f(x) + \Delta g, x' \rangle - f(x') \\
          & = f(x) - \langle \nabla f(x) + \Delta g,x\rangle + f^*(\nabla f(x) + \Delta g).
    \end{aligned}
\end{equation*}
By Young's equality, we have that $f^*(\nabla f(x)) = f(x) - \langle \nabla f(x),x\rangle$ which gives us $    h^*_x(\Delta g) = f^*(\nabla f(x) + \Delta g) - f^*(\nabla f(x)) - \langle x,\Delta g\rangle$. Finally, we obtain from inequality~\eqref{eq:ineq_dual_proof} that condition \eqref{eq:cond_p} is equivalent to
\begin{equation*}
    \forall x \in \mathbb{R}^d, \; \forall \Delta g \in  \mathbb{R}^d, \; f^*(\nabla f(x) + \Delta g) \geqslant f^*(\nabla f(x)) + \langle x,\Delta g\rangle + \frac{1}{2L_{\ell}} \|U_{\ell}^{T}\Delta g\|^2.
\end{equation*}
In particular, for any $x_1,x_2 \in \mathbb{R}^d$, choosing $x = x_2$ and $\Delta g = \nabla f(x_1) - \nabla f(x_2)$ in the previous inequality, we have that:
\begin{equation*}
\begin{aligned}
f^*(\nabla f(x_1)) &\geq f^*(\nabla f(x_2)) + \langle x_2, \nabla f(x_1) - \nabla f(x_2) \rangle \\
&\quad + \frac{1}{2L_{\ell}} \|\nabla^{(\ell)} f(x_1) - \nabla^{(\ell)} f(x_2) \|^2, \quad \forall x_1, x_2 \in \mathbb{R}^d.
\end{aligned}
\end{equation*}
By Young's equality again, we have that:
\begin{equation*}
    f^*(\nabla f(x_1)) = \langle x_1,\nabla f(x_1)\rangle  - f(x_1),
\end{equation*}
\begin{equation*}
    f^*(\nabla f(x_2)) = \langle x_2,\nabla f(x_2)\rangle  - f(x_2).
\end{equation*}
Injecting this two inequalities in the previous inequality, we find that:
\begin{equation*}
\begin{aligned}
    f(x_2) &\geq f(x_1) + \langle \nabla f(x_1), x_2 - x_1 \rangle \\
    &+ \frac{1}{2L_{\ell}} \|\nabla^{(\ell)} f(x_1) - \nabla^{(\ell)} f(x_2) \|^2, \quad \forall x_1, x_2 \in \mathbb{R}^d.
\end{aligned} 
\end{equation*}
which concludes the proof.
\qed

\subsection*{A.2 Proof of Theorem~\ref{th:coord_conds}}
 The implication \( (1) \implies (2) \) follows from Lemma~\ref{lm:coord_upper_bound}, and \( (2) \implies (3) \) is given by Lemma~\ref{lm:coord_lower_bound}. To complete the proof, we show \( (3) \implies (1) \). Assume \( f \) satisfies condition \( (3) \). Then, for all \( x_1, x_2 \in \mathbb{R}^d \),
\begin{equation*}
    f(x_2) \geq f(x_1) + \langle \nabla f(x_1), x_2 - x_1 \rangle.
\end{equation*}
Now, applying condition \( (3) \) with \( x_2 = x + U_{\ell} h^{(\ell)} \) and $x_1 = x$, for any \( x \in \mathbb{R}^d \) and \( h^{(\ell)} \in \mathbb{R}^{d_{\ell}} \), we get:
\begin{equation*}
    f(x + U_{\ell} h^{(\ell)}) \geq f(x) + \langle \nabla f(x), U_{\ell} h^{(\ell)} \rangle + \frac{1}{2L_{\ell}} \|\nabla^{(\ell)} f(x) - \nabla^{(\ell)} f(x + U_{\ell} h^{(\ell)})\|^2.
\end{equation*}
Reversing the roles of \( x \) and \( x + U_{\ell} h^{(\ell)} \) yields:
\begin{equation*}
\begin{aligned}
f(x) &\geq f(x + U_{\ell} h^{(\ell)}) + \langle \nabla f(x + U_{\ell} h^{(\ell)}), -U_{\ell} h^{(\ell)} \rangle \\ 
&+ \frac{1}{2L_{\ell}} \|\nabla^{(\ell)} f(x) - \nabla^{(\ell)} f(x + U_{\ell} h^{(\ell)})\|^2.
\end{aligned}
\end{equation*}
Adding these two inequalities gives:
\begin{equation*}
    \frac{1}{L_{\ell}} \|\nabla^{(\ell)} f(x) - \nabla^{(\ell)} f(x + U_{\ell} h^{(\ell)})\|^2 \leq \langle \nabla f(x + U_{\ell} h^{(\ell)}) - \nabla f(x), U_{\ell} h^{(\ell)} \rangle.
\end{equation*}
By the Cauchy-Schwarz inequality, we have:
\begin{equation*}
    \langle \nabla f(x + U_{\ell} h^{(\ell)}) - \nabla f(x), U_{\ell} h^{(\ell)} \rangle \leq \| \nabla f(x + U_{\ell} h^{(\ell)}) - \nabla f(x) \| \| U_{\ell} h^{(\ell)} \|,
\end{equation*}
which implies:
\begin{equation*}
    \|\nabla^{(\ell)} f(x + U_{\ell} h^{(\ell)}) - \nabla^{(\ell)} f(x)\| \leq L_{\ell} \| U_\ell h^{(\ell)} \|.
\end{equation*}
Since \( \| U_{\ell} h^{(\ell)} \| = \| h^{(\ell)} \| \), this proves \( f \in \mathcal{F}^{\text{coord}}_{0,\mathbf{L}}(\mathbb{R}^d) \).\qed

\subsection*{A.3 Proof of Theorem~\ref{th:coord_cond_suff_2}}
We start by proving that our interpolation conditions are necessary and sufficient for the sets of cardinality $2$. The necessity follows from Theorem~\ref{th:interp_coord}. We establish sufficiency by constructing an explicit example of an interpolating function defined via its Fenchel conjugate. Note that the following propositions are equivalent:
\begin{enumerate}[itemsep=10pt]
    \item Function \( f \) interpolates the set \( \{(x_i, g_i, f_i)\}_{i=1,2} \).
    \item The Fenchel conjugate \( f^* \) of function \( f \) interpolates the set \( \{(g_i, x_i, \langle x_i, g_i \rangle - f_i)\}_{i=1,2} \).
\end{enumerate}
This follows from Young's equality. We assume, without loss of generality, that
\begin{equation*}
    f_1 - f_2 - \langle g_2, x_1 - x_2 \rangle \geq f_2 - f_1 - \langle g_1, x_2 - x_1 \rangle.
\end{equation*}
Since the problem is symmetric in \((f_1,g_1,x_1)\) and \((f_2,g_2,x_2)\), the remainder of the proof remains valid by switching the roles of the variables if the previous inequality is reversed. Consider the function:
\begin{equation}
    f^*(g) = 
    \begin{cases}
        \phi(\lambda) & \text{if } g = \lambda g_2 + (1 - \lambda) g_1, \; \lambda \in [0,1], \\
        \infty & \text{otherwise}
    \end{cases}
\end{equation}
where:
\begin{equation*}
    \phi(\lambda) = (f_1 - f_2 - \langle g_2, x_1 - x_2 \rangle) \lambda^2 + \langle x_1, g_2 - g_1 \rangle \lambda + \langle g_1, x_1 \rangle - f_1.
\end{equation*}
This function is only finite on the segment between $g_1$ and $g_2$. It is easy to verify that $f^*(g_1) = \phi(0) = \langle g_1, x_1 \rangle - f_1, \; f^*(g_2) = \phi(1) = \langle g_2, x_2 \rangle - f_2$. Let us now prove that \( x_1 \) and \( x_2 \) are subgradients of $f^*$ at the points \( g_1 \) and \( g_2 \), respectively which means that for any $g \in \mathbb{R}^d$
\begin{equation*}
    f^*(g) \geq f^*(g_1) + \langle x_1, g-g_1 \rangle,
\end{equation*}
\begin{equation*}
    f^*(g) \geq f^*(g_2) + \langle x_2, g-g_2 \rangle.
\end{equation*}
Consider the first case where \( g \in \mathbb{R}^d \) and \( g \neq \lambda g_2 + (1 - \lambda) g_1 \) for any \( \lambda \in [0,1] \). In this case, we have \( f^*(g) = + \infty \), and the subgradient inequalities are trivially satisfied. Now, consider the case where \( g = \lambda g_2 + (1 - \lambda) g_1 \). Since $(f_1 - f_2 - \langle g_2, x_1 - x_2 \rangle) \geq 0$, \( \phi \) is a convex quadratic function and we have $\phi(\lambda) \geq \phi(0) +\phi'(0) \lambda$ which is equivalent to $f^*(g) \geq f^*(g_1) + \langle x_1, g_2 - g_1 \rangle \lambda$. Simple computations based on $g = \lambda g_2 + (1-\lambda) g_1$ give $\langle x_1, g_2 - g_1 \rangle \lambda = \langle x_1 , g - g_1 \rangle$, which then leads to $f^*(g) \geq f^*(g_1) + \langle x_1, g - g_1 \rangle$. Thus, \( x_1 \) is a subgradient of $f^*$ at \( g_1 \). Similarly, from the convexity of \( \phi \), we obtain $\phi(\lambda) \geq \phi(1) + \phi'(1) (\lambda -1)$, which is equivalent to:
\begin{equation*}
    f^*(g) \geq f^*(g_2) +  ( 2 (f_1 - f_2 - \langle g_2, x_1 - x_2 \rangle)  + \langle x_1, g_2 - g_1 \rangle ) (\lambda - 1).
\end{equation*}
Since we assume that $f_1 - f_2 - \langle g_2, x_1 - x_2 \rangle \geq f_2 - f_1 - \langle g_1, x_2 - x_1 \rangle$, it follows by adding $f_1 - f_2 - \langle g_2, x_1 - x_2 \rangle$ on both sides that:
\begin{equation}
    2 (f_1 - f_2 - \langle g_2, x_1 - x_2 \rangle) \geq \langle x_2 - x_1 , g_2 - g_1 \rangle.
\end{equation}
This implies $f^*(g) \geq f^*(g_2) + ( \langle x_2 -x_1, g_2 - g_1 + \rangle + \langle x_1,g_2-g_1\rangle) (\lambda - 1 )$, which is equivalent to $f^*(g) \geq f^*(g_2) + \langle x_2, g_2 - g_1 \rangle (\lambda - 1)$. Since $\langle x_2, g_2 - g_1 \rangle (\lambda - 1) = \langle x_2, g - g_2 \rangle$, we obtain $f^*(g) \geq f^*(g_2) + \langle x_2, g - g_2 \rangle$, and \( x_2 \) is a subgradient at \( g_2 \). Hence function \( f^* \) interpolates the set \( \{(g_i, x_i, \langle x_i, g_i \rangle - f_i)\}_{i=1,2} \), which implies that \( f = (f^*)^* \), defined as:
\begin{equation*}
    f(x) = \sup_{g \in \mathbb{R}^d} \langle x, g \rangle - f^*(g), \quad \forall x \in \mathbb{R}^d,
\end{equation*}
also interpolates the set \( \{(x_i, g_i, f_i)\}_{i=1,2} \). By definition of $f^*$, \( f \) takes finite values for all \( x \in \mathbb{R}^d \), and for any $x \in \mathbb{R}^d$, there exists \( g_x = \lambda_x g_2 + (1 - \lambda_x) g_1 \in \mathbb{R}^d \) such that:
\begin{equation}\label{eq:young_coord}
    f(x) = \langle x, g_x \rangle - f^*(g_x).
\end{equation}
By Young's equality, we have \( g_x \in \partial f(x) \) and \( x \in \partial f^*(g_x) \). To complete the proof, it remains to show that \( f \in \mathcal{F}^{\text{coord}}_{0,\mathbf{L}}(\mathbb{R}^d) \). To do so, we first prove that the functions \( g \mapsto f^*(g) - \frac{1}{2L_\ell} \|g^{(\ell)}\|^2 \) are convex and show that this implies \( f \in \mathcal{F}^{\text{coord}}_{0,\mathbf{L}}(\mathbb{R}^d) \). Simple computations yield:
\begin{equation*}
    h_{\ell}(g) = f^*(g) - \frac{1}{2L_\ell} \|g^{(\ell)}\|^2 = 
    \begin{cases}
        \psi_{\ell}(\lambda), & \text{if } g = \lambda g_2 + (1 - \lambda) g_1, \; \lambda \in [0,1]\\
        + \infty, & \text{otherwise},
    \end{cases}
\end{equation*}
where:
\begin{align*}
    \psi_{\ell}(\lambda) = \left(f_1 - f_2 - \langle g_2, x_1 - x_2 \rangle - \frac{1}{2L_\ell} \|g_2^{(\ell)} - g_1^{(\ell)}\|^2\right) \lambda^2 & \\
    + \left(\langle x_1 - \frac{1}{L_\ell} U_\ell g_1^{(\ell)}, g_2 - g_1 \rangle\right) \lambda & \\
    + \langle g_1, x_1 \rangle - f_1 - \frac{1}{2L_\ell} \|g_1^{(\ell)}\|^2.
\end{align*}
Since the set \( \{(x_i, g_i, f_i)\}_{i=1,2} \) satisfies the interpolation conditions~\eqref{eq:coord_interp_conds}, the term  $f_1 - f_2 - \langle g_2, x_1 - x_2 \rangle - \frac{1}{2L_\ell} \|g_2^{(\ell)} - g_1^{(\ell)}\|^2$ is nonnegative. Thus, \( \psi_{\ell} \) is a convex quadratic function, implying that \( h_{\ell} \) is convex. For all \( x \in \mathbb{R}^d \) and \( g_x \) as defined in~\eqref{eq:young_coord}, we have that if $g_x \notin \{g_1,g_2\}$, $f^*$ is differentiable and \( x - \frac{1}{L_{\ell}} U_{\ell} g_x^{(\ell)} \) is the gradient of \( h_{\ell}\) at \( g_x \). Moreover, $x_{1} - \frac{1}{L_{\ell}} U_{\ell} g_{1}^{(\ell)}$ and $x_{2} - \frac{1}{L_{\ell}} U_{\ell} g_{2}^{(\ell)}$ are respectively subgradients of  \( h_{\ell}\) at $g_1$ and $g_2$. Indeed, consider the first case where \( g \in \mathbb{R}^d \) and \( g \neq \lambda g_2 + (1 - \lambda) g_1 \) for any \( \lambda \in [0,1] \). In this case, we have \( h_{\ell}(g) = + \infty \), and the subgradient inequalities
\begin{equation*}
    h_{\ell}(g) \geq h_{\ell}(g_1) + \langle x_1 - \frac{1}{L_{\ell}} U_{\ell} g_{1}^{(\ell)}, g-g_1 \rangle,
\end{equation*}
\begin{equation*}
    h_{\ell}^*(g) \geq h_{\ell}(g_2) + \langle x_2 - \frac{1}{L_{\ell}} U_{\ell} g_{2}^{(\ell)}, g-g_2 \rangle
\end{equation*}
are trivially satisfied. Now, consider the case where \( g = \lambda g_2 + (1 - \lambda) g_1 \). Since \( \phi \) is a convex quadratic function, we have $\psi_{\ell}(\lambda) \geq \psi_{\ell}(0) + \psi_{\ell}'(0) \lambda$, which is equivalent to $h_{\ell}(g) \geq h_{\ell}(g_1) + \langle x_1- \frac{1}{L_{\ell}} U_{\ell} g_{1}^{(\ell)}, g_2 - g_1 \rangle \lambda$. Simple computations give $\langle x_1 - \frac{1}{L_{\ell}} U_{\ell} g_{1}^{(\ell)}, g_2 - g_1 \rangle \lambda = \langle x_1 - \frac{1}{L_{\ell}} U_{\ell} g_{1}^{(\ell)} , g - g_1 \rangle$, which then leads to $h_{\ell}(g) \geq h_{\ell}(g_1) + \langle x_1 - \frac{1}{L_{\ell}} U_{\ell} g_{1}^{(\ell)}, g - g_1 \rangle$. Thus, \( x_1- \frac{1}{L_{\ell}} U_{\ell} g_{1}^{(\ell)} \) is a subgradient of $h_{\ell}$ at \( g_1 \). Similarly, from the convexity of \( \psi_{\ell} \), we obtain $\psi_{\ell}(\lambda) \geq \psi_{\ell}(1) + \psi_{\ell}'(1) (\lambda -1)$, which is equivalent to:
\begin{equation*}
\begin{aligned}
h_{\ell}(g) \geq h_{\ell}(g_2) +  ( 2 (f_1 - f_2 - \langle g_2, x_1 - x_2 \rangle - \frac{1}{2L_\ell} \|g_2^{(\ell)} - g_1^{(\ell)}\|^2 ) & + \\ \langle x_1 - \frac{1}{L_{\ell}} U_{\ell} g_{1}^{(\ell)}, g_2 - g_1 \rangle ) (\lambda - 1).
\end{aligned}
\end{equation*}
Since we assume that $f_1 - f_2 - \langle g_2, x_1 - x_2 \rangle \geq f_2 - f_1 - \langle g_1, x_2 - x_1 \rangle$, it follows that $2 (f_1 - f_2 - \langle g_2, x_1 - x_2 \rangle) \geq \langle x_2 - x_1 , g_2 - g_1 \rangle$, which implies that 
\begin{equation*}
    \begin{aligned}
       2 (f_1 - f_2 - \langle g_2, x_1 - x_2 \rangle - \frac{1}{2L_\ell} \|g_2^{(\ell)} - g_1^{(\ell)}\|^2 ) + \langle x_1 - \frac{1}{L_{\ell}} U_{\ell} g_{1}^{(\ell)}, g_2 - g_1 \rangle & \geq \\ \langle x_2 - x_1 , g_2 - g_1 \rangle  -  \frac{1}{2L_\ell} \|g_2^{(\ell)} - g_1^{(\ell)}\|^2 +  \langle x_1 - \frac{1}{L_{\ell}} U_{\ell} g_{1}^{(\ell)}, g_2 - g_1 \rangle & = \\ \langle x_2 - \frac{1}{L_{\ell}} U_{\ell} g_{2}^{(\ell)}, g_2 - g_1\rangle
    \end{aligned}
\end{equation*}
which gives us $    h_{\ell}(g) \geq h_{\ell}(g_2) + \langle x_2 - \frac{1}{L_{\ell}} U_{\ell} g_{2}^{(\ell)}, g_2 - g_1 \rangle (\lambda - 1)$. Since $     \langle  x_2 - \frac{1}{L_{\ell}} U_{\ell} g_{2}^{(\ell)} , g_2 - g_1 \rangle (\lambda - 1) =  \langle  x_2 - \frac{1}{L_{\ell}} U_{\ell} g_{2}^{(\ell)} , g - g_2 \rangle$, we obtain $   h_{\ell}(g) \geq h_{\ell}(g_2) + \langle x_2 - \frac{1}{L_{\ell}} U_{\ell} g_{2}^{(\ell)}, g - g_2 \rangle$, which proves that $x_2 - \frac{1}{L_{\ell}} U_{\ell} g_{2}^{(\ell)}$ is a subgradient of $h_{\ell}$ at $g_2$.
\noindent By the convexity of $h_{\ell}$, for all \( \ell \in \{1,\dots,p\} \) and for all \( x,y \in \mathbb{R}^d \), we obtain:
\begin{equation*}
    f^*(g_y) - \frac{1}{2L_\ell} \|g_y^{(\ell)}\|^2 \geq f^*(g_x) - \frac{1}{2L_\ell} \|g_x^{(\ell)}\|^2 + \left\langle x - \frac{1}{L_\ell} U_\ell g_x^{(\ell)}, g_y - g_x \right\rangle.
\end{equation*}
By substituting the equalities $f(x) + f^*(g_x) = \langle x,g_x\rangle$ and $ f(y) + f^*(g_y) = \langle y,g_y\rangle$, the previous inequality becomes, for all \( \ell \in \{1,\dots,p\} \) and for all \( x,y \in \mathbb{R}^d \):
\begin{equation*}
    f(x) - f(y) \geq \langle g_y, x-y \rangle + \frac{1}{2L_\ell} \left( \|g_y^{(\ell)}\|^2 - \|g_x^{(\ell)}\|^2 - 2 \langle U_\ell g^{(\ell)}_x , g_y - g_x\rangle  \right).
\end{equation*}
From the definition of the selection matrix \( U_{\ell} \), we have $   \langle U_\ell g^{(\ell)}_x , g_y \rangle = \langle g^{(\ell)}_x, g^{(\ell)}_y \rangle$ and $     \langle U_\ell g^{(\ell)}_x , g_x \rangle = \| g^{(\ell)}_x\|^2$. Thus, for all \( \ell \in \{1,\dots,p\} \) and for all \( x,y \in \mathbb{R}^d \), we obtain:
\begin{equation*}
    f(x) - f(y) \geq \langle g_y, x-y \rangle + \frac{1}{2L_\ell} \|g^{(\ell)}_x - g^{(\ell)}_y\|^2.
\end{equation*}
Since \( g_x \in \partial f(x) \) and \( g_y \in \partial f(y) \), Theorem~\ref{th:coord_conds} ensures that \( f \) belongs to \( \mathcal{F}^{\text{coord}}_{0,\textbf{L}}(\mathbb{R}^d) \), which concludes the proof of the sufficiency of our conditions for sets of cardinality equal to $2$. We now provide a counterexample of a set of triplets of cardinality equal to $3$ that satisfies the interpolation conditions~\eqref{eq:coord_interp_conds} but is not $\mathcal{F}^{\text{coord}}_{0,\textbf{L}}(\mathbb{R}^d)$-interpolable. Consider the set:
\begin{table}[H]
    \centering
    \resizebox{0.7\textwidth}{!}{%
    \begin{tabular}{||c||c||c||}
    \hline
    Points $x_i$ &  Gradients $g_i$ & Function Values $f_i$  \\
    \hline \hline
    $x_1 = (-1,0)^T$ & $g_1 = (-1,0)^T$ & $f_1 = \frac{1}{2}$ \\
    \hline
    $x_2 = (0,0)^T$ &  $g_2 = (0,-1)^T$ & $f_2 = 0$ \\
    \hline
    $x_3 = (1,0)^T$ & $g_3 = (1,0)^T$ &  $f_3 = \frac{1}{2}$  \\
     \hline
    \end{tabular}%
    }
    \caption{\small A set $\mathcal{S}$ satisfying~\eqref{eq:coord_interp_conds} but not interpolable.}
    \label{tab:coord_interp_example}
\end{table}

\noindent The set $\mathcal{S}$ satisfies~\eqref{eq:coord_interp_conds} for $\textbf{L} = (1,1)$. Suppose that there exists a function $f \in \mathcal{F}^{\text{coord}}_{0,\textbf{L}}(\mathbb{R}^2)$ that interpolates $\mathcal{S}$. For any point $x = (-1,y)$, the convexity of $f$ implies that $f(x)  \geqslant f(x_1) + \langle g_1, x-x_1\rangle$, which is equivalent to $f(-1,y) \geq \frac{1}{2}$. Similarly for the points of the form $x = (1,y)$, we have that $f(x) \geq f(x_3) + \langle g_3, x-x_3\rangle$, which is equivalent to $f(1,y) \geq \frac{1}{2}$. By $1$-smoothness of $f$ along the second coordinate, for any point $x = (0,y)$ we have $    f(0,y) \leqslant f(0,0) + g_2^{(2)} y + \frac{y^2}{2} = - y + \frac{y^2}{2}$. Setting $y = 1$, we obtain from all the above:
\begin{equation*}
    f(-1,1) \geqslant \frac{1}{2}, \quad f(0,1) \leqslant -\frac{1}{2}, \quad f(1,1) \geqslant \frac{1}{2}.
\end{equation*}
Define $h(x) = f(x,1)$. Then $h$ is a $1$-smooth convex function satisfying:
\begin{equation*}
    h(-1) \geqslant \frac{1}{2}, \quad h(0) \leqslant -\frac{1}{2}, \quad h(1) \geqslant \frac{1}{2}.
\end{equation*}
Suppose $h'(0) \leq 0$, then since $h$ is a $1$-smooth convex function, we have that
\begin{equation*}
    h(1) \leqslant h(0) + h'(0)(1-0) + \frac{(1-0)^2}{2} \leqslant h(0) + \frac{1}{2} \leqslant 0,
\end{equation*}
which contradicts $h(1) \geqslant \frac{1}{2}$. Similarly, if $h'(0) \geqslant 0$, we have that:
\begin{equation*}
        h(-1) \leqslant h(0) + h'(0)(-1-0) + \frac{(-1-0)^2}{2} \leqslant h(0) + \frac{1}{2} \leqslant 0,
\end{equation*}
which contradicts $h(-1) \geqslant \frac{1}{2}$. Thus, an interpolating function $f \in \mathcal{F}^{\text{coord}}_{0,\textbf{L}}(\mathbb{R}^d)$ for the set $\mathcal{S}$ cannot exist.
\qed 
\end{document}